\documentclass[]{article}

\usepackage{amsmath}
\usepackage{amsfonts}
\usepackage{amsthm}
\usepackage[letterpaper]{geometry}
\usepackage[]{hyperref}

\numberwithin{equation}{section}

\title{Correlation Functions of the Schur Process Through Macdonald Difference Operators}
\author{Amol Aggarwal} 

\begin{document}

\maketitle

\begin{abstract}
Introduced by Okounkov and Reshetikhin, the Schur process is known to be a determinantal point process, meaning that its correlation functions are minors of a single correlation kernel matrix. Previously, this was derived using determinantal expressions for the skew-Schur polynomials. In this paper we obtain this result in a different way, using the fact that the Schur polynomials are eigenfunctions of Macdonald difference operators. 
\end{abstract}

\newtheorem{thm}{Theorem}[subsection]
\newtheorem{prop}[thm]{Proposition} 
\newtheorem{lem}[thm]{Lemma}
\newtheorem{cor}[thm]{Corollary}

\theoremstyle{remark}

\newtheorem{rem}[thm]{Remark}

\section{Introduction}

\subsection{Background and Results} 

A {\itshape partition} $\lambda = (\lambda_1, \lambda_2, \ldots )$ is a nonincreasing sequence of nonnegative integers such that $\sum_{i = 1}^{\infty} \lambda_i$ is finite; this sum is called the {\itshape size} of $\lambda$ and is denoted by $|\lambda|$. The number of positive $\lambda_i$ ({\itshape parts} of $\lambda$) is called the {\itshape length} of $\lambda$ and is denoted by $\ell (\lambda)$. For each positive integer $i$, let $m_i (\lambda)$ denote the number of parts of $\lambda$ equal to $i$. For each nonnegative integer $n$, let $\mathbb{Y}_n$ denote the set of partitions of size $n$, and let $\mathbb{Y} = \bigcup_{n = 0}^{\infty} \mathbb{Y}_n$ denote the set of all partitions. If $\lambda, \mu \in \mathbb{Y}$ satisfy $\mu_i \le \lambda_i$ for all positive integers $i$, we say that $\mu \subseteq \lambda$ or equivalently $\lambda \supseteq \mu$. 

Let $X$ and $Y$ be (possibly infinite) sets of variables. For any partitions $\lambda, \mu \in \mathbb{Y}$, let $s_{\lambda} (X)$ denote the Schur polynomial in $X$ associated with $\lambda$ (if $|X| < \ell (\lambda)$, then we set $s_{\lambda} (X) = 0$); let $s_{\lambda / \mu} (X)$ denote the skew-Schur polynomial associated with $\lambda$ and $\mu$; and let $F(X; Y)$ denote the Cauchy product $\prod_{(x, y) \in X \times Y} (1 - xy)^{-1}$.

Define the measure {\bfseries SM} on $\mathbb{Y}$ by setting
\begin{flalign*}
\textbf{SM} (\lambda) = \textbf{SM}_{X, Y} ( \{ \lambda \} ) = \displaystyle\frac{s_{\lambda} (X) s_{\lambda} (Y)}{F(X; Y)} 
\end{flalign*}

\noindent for all $\lambda \in \mathbb{Y}$. Originally introduced by Okounkov in \cite{23}, the measure $\textbf{SM}$ is called the {\itshape Schur measure}. The {\itshape Cauchy identity}
\begin{flalign}
\label{sum}
\displaystyle\sum_{\lambda \in \mathbb{Y}} s_{\lambda} (X) s_{\lambda} (Y) = F(X; Y)
\end{flalign}

\noindent implies that $\sum_{\lambda \in \mathbb{Y}} \textbf{SM} (\lambda) = 1$. Furthermore, a combinatorial interpretation of the Schur functions yields that $s_{\lambda} (X)$ and $s_{\lambda} (Y)$ are nonnegative for each $\lambda \in \mathbb{Y}$ if each element of $X$ and $Y$ is a nonnegative real number. Therefore, $\textbf{SM}_{X, Y}$ is a probability measure if $X$ and $Y$ are finite sets of nonnegative numbers less than $1$. 

Okounkov and Reshetikhin later generalized the Schur measure by defining the Schur process \cite{25}. For any positive integer $m$, let $\lambda = \{ \lambda^{(1)}, \lambda^{(2)}, \ldots , \lambda^{(m)} \}$ and $\mu = \{\mu^{(1)}, \mu^{(2)}, \ldots , \mu^{(m - 1)} \}$ be sequences of partitions. Let $X^{(i)}$ and $Y^{(i)}$ be (possibly infinite) sets of variables for each integer $i \in [1, m]$. Define the product $Z_{X, Y} = \prod_{1\le i\le j\le m} F \big( X^{(i)}; Y^{(j)} \big)$, and define the weight function
\begin{flalign}
\label{weight}
\mathcal{W}_{X, Y} (\lambda, \mu) = s_{\lambda^{(1)}} \big(X^{(1)} \big) \left( \displaystyle\prod_{i = 1}^{m-1} s_{\lambda^{(i+1)} / \mu^{(i)}} \big(X^{(i + 1)} \big) s_{\lambda^{(i)} / \mu^{(i)}} \big(Y^{(i)} \big) \right) s_{\lambda^{(m)}} \big( Y^{(m)} \big). 
\end{flalign}

\noindent Now define the measure {\bfseries S} on $\mathbb{Y}^m \times \mathbb{Y}^{m - 1}$ by setting
\begin{flalign*}
\textbf{S} (\lambda, \mu) = \textbf{S}_{X, Y} ( \{ (\lambda, \mu) \}) = \displaystyle\frac{\mathcal{W}_{X, Y} (\lambda, \mu)}{Z_{X, Y}}
\end{flalign*}

\noindent for all $\lambda \in \mathbb{Y}^m$ and $\mu \in \mathbb{Y}^{m - 1}$. Since the integer $m$ can be viewed as a discrete time parameter, the measure $\textbf{S}$ is called the {\itshape Schur process}. 

Observe that $\textbf{S}$ is supported on pairs $(\lambda, \mu)$ satisfying $\lambda^{(1)} \supseteq \mu^{(1)} \subseteq \lambda^{(2)} \supseteq \mu^{(2)} \subseteq \cdots \subseteq \lambda^{(m)}$ and $\ell (\lambda^{(i)}) \le \max \{ |X^{(i)}|, |Y^{(i)}| \}$, for each integer $i \in [1, m]$. The ``generalized Cauchy identity" (see Proposition 6.2 of \cite{11} for a proof) 
\begin{flalign*}
\displaystyle\sum_{(\lambda, \mu) \in \mathbb{Y}^m \times \mathbb{Y}^{m - 1}} \mathcal{W}_{X, Y} (\lambda, \mu) = Z_{X, Y}, 
\end{flalign*}

\noindent implies that $\sum_{(\lambda, \mu) \in \mathbb{Y}^m \times \mathbb{Y}^{m - 1}} \textbf{S} (\lambda, \mu) = 1$. Furthermore, a combinatorial interpretation of the skew-Schur functions implies that $\mathcal{W}_{X, Y} (\lambda, \mu)$ is nonnegative for all $(\lambda, \mu) \in \mathbb{Y}^m \times \mathbb{Y}^{m - 1}$ if each element of $X^{(i)}$ and $Y^{(i)}$ is a nonnegative real number for all integers $i \in [1, m]$. Therefore, $\textbf{S}_{X, Y}$ is a probability measure when each of the $X^{(i)}$ and $Y^{(i)}$ are finite sets of nonnegative numbers less than $1$. The Schur process may be projected onto $\mathbb{Y}^m$, giving weight $\textbf{S} (\lambda) = \sum_{\mu \in \mathbb{Y}^{m - 1}} \textbf{S} (\lambda, \mu)$ to each $\lambda \in \mathbb{Y}^m$. Observe that the Schur measure is the special $m = 1$ case of the Schur process. 

The Schur measure and Schur process are both known to specialize to probability measures that are useful in combinatorics, probability, and mathematical physics. For instance, it is shown in \cite{23} that the Schur measure specializes to the {\itshape Poissonized Plancherel measure} $\textbf{PP}_{\theta}$ (where $\theta \in \mathbb{R}$ is some parameter), which gives weight $\textbf{PP}_{\theta} (\lambda) = e^{-\theta |\lambda|} \dim (\lambda)^2 / |\lambda|!^2$ to each $\lambda \in \mathbb{Y}$. The Poissonized Plancherel measure is known to be related to many random growth models, including polynuclear growth, Last Passage Percolation, and the length of the longest increasing subsequence of a random permutation (see \cite{1, 11, 15, 16} and references therein). In addition to these examples, the Schur measure and Schur process have been used to understand a wide variety of other combinatorial processes, including plane partitions, lozenge tilings, Aztec tilings, random words, and the Totally Asymmetric Simple Exclusion Process (see \cite{4, 9, 11, 12, 17, 18, 25} and references therein).

However, in order to obtain results about these processes, one requires a refined analysis of the Schur measure and Schur process. This can be done by finding exact forms for their correlation functions. Let us first define the correlation functions of the Schur measure. Let $\mathcal{X}$ be the function mapping $\mathbb{Y}$ to finite subsets of $\mathbb{Z}$ that sends any partition $\lambda = (\lambda_1, \lambda_2, \ldots ) \in \mathbb{Y}$ to the subset $\mathcal{X}(\lambda) = \{ \lambda_1 - 1, \lambda_2 - 2, \lambda_3 - 3, \ldots , \lambda_{\ell(\lambda)} - \ell(\lambda) \} \subset \mathbb{Z}$. For any finite subset $T \subset \mathbb{Z}$, we define the {\itshape correlation function} $\rho_{\textbf{SM}} (T)$ to be the probability that $T \subseteq \mathcal{X}(\lambda)$ when $\lambda$ is randomly chosen under the Schur measure. Equivalently, 
\begin{flalign}
\label{correlationmeasure}
\rho_{\textbf{SM}} (T) = \displaystyle\sum_{\lambda \in \mathbb{Y}} \textbf{1}_{T \subseteq \mathcal{X} (\lambda)} \textbf{SM} (\lambda),
\end{flalign}

\noindent where $\textbf{1}_E$ is the indicator function for an event $E$. 

To define the analogue for the Schur process, let $\mathfrak{S}$ be the function mapping $\mathbb{Y}^m$ to finite subsets of $\{ 1, 2, \ldots , m \} \times \mathbb{Z}$ that sends any sequence of partitions $\lambda = \{ \lambda^{(1)}, \lambda^{(2)}, \ldots , \lambda^{(m)} \}$ to the subset $\{ (i, \lambda^{(i)}_j - j) \} \subset \{ 1, 2, \ldots , m \} \times \mathbb{Z}$, where $i$ ranges from $1$ to $m$; $j$ ranges from $1$ to $\ell (\lambda^{(i)})$; and $\lambda^{(i)} = \{ \lambda^{(i)}_1, \lambda^{(i)}_2, \ldots \}$ for each integer $i \in [1, m]$. For any finite subset $T \subset \{ 1, 2, \ldots , m \} \times \mathbb{Z}$, we define {\itshape correlation function} $\rho_{\textbf{S}} (T)$ to be the probability that $T \subseteq \mathfrak{S} (\lambda)$ when $\lambda \in \mathbb{Y}^m$ is randomly chosen under the Schur process. Equivalently, 
\begin{flalign}
\label{correlationprocess}
\rho_{\textbf{S}} (T) = \displaystyle\sum_{\lambda \in \mathbb{Y}^m} \textbf{1}_{T \subseteq \mathfrak{S} (\lambda)} \textbf{S} (\lambda). 
\end{flalign} 

The Schur measure and Schur process are amenable to asymptotic analysis because they are determinantal point processes, meaning that their correlation functions are minors of a single {\itshape correlation kernel matrix} (see, for example, \hyperref[measure]{Theorem \ref*{measure}} and \hyperref[process]{Theorem \ref*{process}}); we refer to the survey \cite{3} for more information about determinantal point processes. In particular, if the entries of the correlation kernel matrices associated with the Schur measure and Schur process are suitable to asymptotic analysis, then one might be able to understand the asymptotics of the correlation functions $\rho_{\textbf{SM}}$ and $\rho_{\textbf{S}}$; this can lead to results about some of the processes discussed above. Chapter 5 of \cite{11} shows how to use this method to analyze Last Passage Percolation (see also \cite{14, 24} for more information on asymptotic methods). 

In this paper we establish the following two theorems, which give explicit forms for the correlation kernel matrices associated with the Schur measure and Schur process. In the below, $S^{-1}$ refers to the set $\{ s^{-1} \}_{s \in S}$ for any subset $S \subset \mathbb{C} \backslash \{ 0 \}$. 
 	
 \begin{thm}
 \label{measure} 
 
 Let $X$ and $Y$ be finite sets of nonnegative numbers less than $1$. For each $i, j \in \mathbb{Z}$, let 
 \begin{flalign}
 \label{kernelmeasure} 
 L (i, j) = \displaystyle\frac{1}{4\pi^2} \oint \oint \displaystyle\frac{1}{w-z} \left( \displaystyle\frac{F (Y; \{ w^{-1} \} ) F(X; \{ z \})}{F(Y; \{ z^{-1} \}) F(X; \{ w \})} \right) w^{j} z^{- i - 1} dw dz, 
 \end{flalign}

 \noindent where the contours are taken along the positively oriented circles $|z| = r_1$ and $|w| = r_2$, where $r_1$ and $r_2$ are any positive reals satisfying $\max X, \max Y < r_2 < r_1 < \min (X^{-1}), \min (Y^{-1})$. Then, $\rho_{\textbf{SM}} (T) = \det \textbf{L}_T$ for any finite subset $T = \{t_1, t_2, \ldots t_d \} \subset \mathbb{Z}$, where $\textbf{L}_T$ is the $d \times d$ matrix whose $(r, c)$ entry is $L (t_r, t_c)$. 
 \end{thm}

 \begin{thm}
 \label{process}
 
Let $X^{(1)}, X^{(2)}, \ldots , X^{(m)}$ and $Y^{(1)}, Y^{(2)}, \ldots , Y^{(m)}$ be finite sets of nonnegative numbers less than $1$. For each $(s, i), (t, j) \in \{ 1, 2, \ldots , m \} \times \mathbb{Z}$, let 
\begin{flalign}
\label{kernelprocess}
K(s, i; t, j) = \displaystyle\frac{1}{4\pi^2} \displaystyle\oint \displaystyle\oint \displaystyle\frac{1}{w-z} \left( \displaystyle\frac{\prod_{k = t}^m F \big( Y^{(k)}; \{ w^{-1} \} \big) \prod_{k = 1}^s F \big( X^{(k)}; \{ z \} \big) }{\prod_{k = s}^m F \big( Y^{(k)}; \{ z^{-1} \} \big) \prod_{k = 1}^t F \big(X^{(k)}; \{ w \} \big) }\right) w^{j} z^{- i - 1} dw dz, 
\end{flalign}

\noindent where the contours are taken along the positively oriented circles $|z| = r_1$ and $|w| = r_2$, where $r_1$ and $r_2$ are any positive reals satisfying the following. If $s > t$, then $\max X^{(i)}, \max Y^{(i)} < r_1 < r_2 < \min \big( (X^{(i)})^{-1} \big), \min \big( (Y^{(i)})^{-1} \big)$ for all integers $i \in [1, m]$; otherwise, $\max X^{(i)}, \max Y^{(i)} < r_2 < r_1 < \min \big( (X^{(i)})^{-1} \big), \min \big( (Y^{(i)})^{-1} \big)$ for all integers $i \in [1, m]$. Then, $\rho_{\textbf{S}} (T) = \det \textbf{K}_T$ for any finite subset $T = \{ (a_1, b_1), (a_2, b_2), \ldots , (a_d, b_d) \} \subset \{1, 2, \ldots , m \} \times \mathbb{Z}$, where $\textbf{K}_T$ is the $d \times d$ matrix whose $(r, c)$ entry is $K(a_r, b_r; a_c, b_c)$. 
\end{thm} 

\begin{rem}
Although not stated above, it is possible to show that both theorems above hold in a formal setting (see Section 2 of \cite{6}), in which the $X^{(i)}$ and $Y^{(i)}$ are infinite sets of formal variables and the contour integrals represent sums of appropriate residues of the integrand. This more general formulation can be necessary for application; for instance, the Poissonized Plancherel measure is not obtained from selecting suitable finite sets $X$ and $Y$ of real numbers in \hyperref[measure]{Theorem \ref*{measure}}. Instead, it is obtained from specializing the Schur measure $\textbf{S}_{X, Y}$ where $X$ and $Y$ are infinite sets of formal variables (see \cite{23} or Chapter 5 of \cite{11}). However, the more general statement can be derived either directly from \hyperref[measure]{Theorem \ref*{measure}} and \hyperref[process]{Theorem \ref*{process}} or from applying our methods in the formal setting given in Section 2 of \cite{6}. For the sake of brevity, we will not pursue this here; instead, we will adhere to the more familiar framework in which the $X^{(i)}$ and $Y^{(i)}$ are finite sets of numbers between $0$ and $1$. 
\end{rem}

Both of the theorems above have been established previously. \hyperref[measure]{Theorem \ref*{measure}} is a special case of \hyperref[process]{Theorem \ref*{process}}, which was originally shown to hold in \cite{25} by Okounkov and Reshetikhin using fermionic Fock space techniques. In particular, they evaluated the correlation functions $\rho_{\textbf{S}}$ explicitly by interpreting them as matrix elements of the fermionic Fock space; this required the use of the Jacobi-Trudi identity, which is a determinantal expression for the skew-Schur functions. Another proof of \hyperref[process]{Theorem \ref*{process}} was found in \cite{13} by Borodin and Rains through the Eynard-Mehta theorem; they also used the Jacobi-Trudi identity. 

In this paper we derive the correlation kernel matrix of the Schur process without using any determinantal identities for the Schur functions. Instead we use the fact that the Schur polynomials are special $q = t$ cases of the Macdonald $(q, t)$-polynomials, which are eigenfunctions of the Macdonald $(q, t)$-difference operators. This idea is becoming increasingly popular in the analysis of Macdonald processes, which in general are not known to exhibit any determinantal behavior (we refer to \cite{5, 6, 7, 8, 10, 11, 12} and references therein for examples and additional information). 

Let us briefly outline our method; in order to prove \hyperref[measure]{Theorem \ref*{measure}}, we use ideas from \cite{5}. Remark 2.2.15 of \cite{5} first suggests to apply the Macdonald $(q, q)$-difference operators in $X$ to the Cauchy product $F(X; Y) = \sum_{\lambda \in \mathbb{Y}} s_{\lambda} (X) s_{\lambda} (Y)$. The results from Chapter 2, Section 2 of \cite{5} can then be used to put the resulting expression in a manageable contour integral form. Since the Schur polynomials are eigenfunctions of the $(q, q)$-Macdonald difference operators, this yields a family of observables (indexed by $q$), each of which has a contour integral form, for the Schur measure. As was also suggested in Remark 2.2.15 of \cite{5}, we then vary $q$ to obtain ``enough" observables for the Schur measure in order to extract the correlation functions $\rho_{\textbf{SM}}$ in a systematic way. 

This does not directly apply to the Schur process because the skew-Schur polynomials are not always eigenfunctions of Macdonald difference operators. Therefore, we first implement a method from \cite{6} that expresses the skew-Macdonald polynomials in terms of scalar products of Macdonald polynomials; this was originally done in \cite{6} in order to evaluate multi-level observables for the Macdonald process. In our case, this allows us to express the skew-Schur functions in terms of scalar products of the Schur functions. We can then apply a method similar to the one used for the Schur measure to obtain a large family of observables for the Schur process, which allows us to evaluate the correlation functions $\rho_{\textbf{S}}$. 

Although the methods used in this article were recently popularized by Borodin and Corwin in \cite{5} for probabilistic reasons, algebraic combinatorialists have been using Macdonald difference operators to obtain identities for symmetric functions since the 1990s. For instance, the fact that the Macdonald polynomials are eigenfunctions of the Macdonald difference operators was used in \cite{19, 26} to deduce the Kirillov-Noumi-Warnaar identity; the $q = t$ case of this identity resembles \hyperref[operatorsmeasure]{Proposition \ref*{operatorsmeasure}} when $q_1 = q_2 = \cdots = q_m$. This fact was also used in \cite{20, 21} to exhibit the Lassalle-Schlosser identity, which is a non-determinantal generalization of the Jacobi-Trudi identity for Macdonald polynomials. More recently, Betea and Wheeler used Macdonald difference operators to obtain several new identities involving symmetric functions and relate them to plane partitions and alternating sign matrices \cite{2}. 

The Schur process is a special case of the Macdonald process, which has recently been a significant topic of research due to its applications in combinatorics, probability, representation theory, and mathematical physics (see \cite{5, 6, 7, 8, 10, 11, 12} and references therein). Therefore, finding asymptotically analyzable expressions for the correlation functions of the Macdonald process would have many implications in these fields. For instance, they would yield the correlation functions for the joint eigenvalue distribution of general $\beta$ random matrix ensembles, which has been of interest to probabilists and mathematical physicists for over fifty years (see \cite{10}). The previous two derivations of the correlation functions of the Schur process in \cite{13, 25} used the Jacobi-Trudi identity; there is no known analogue of this identity that produces determinantal expressions for the Macdonald polynomials, so it seems likely that a new method will be required in order to find the correlation functions of the Macdonald process. Unfortunately, our methods alone are unable to accomplish this task because they do not yield enough observables for the Macdonald process (for instance, the parameters $q$ and $t$ cannot be varied in the general Macdonald setting) for us to extract its correlation functions systematically. 

Still, our method is interesting to us for two reasons. The first reason is linear algebraic. Although the Macdonald $(q, t)$-polynomials and difference operators are not known to exhibit any determinantal behavior, our proof shows that they can still be used to ``detect" that the Schur measure is a determinantal point process. This phenomenon has been corroborated in Chapter 3 of \cite{5}, in which Borodin and Corwin use Macdonald difference operators to show how a Fredholm determinant arises from observables for the Whittaker process; other determinantal results obtained through Macdonald difference operators can be found in \cite{6, 7, 8}. However, our proof is the first to use Macdonald difference operators to find determinantal expressions in the Schur process (see \hyperref[determinants]{Remark \ref*{determinants}}). 

The second reason is combinatorial. As mentioned previously, the proof of \hyperref[process]{Theorem \ref*{process}} by Okounkov and Reshetikhin uses techniques from the fermionic Fock space, which is an object that arises from representation theory; the proof of \hyperref[process]{Theorem \ref*{process}} by Borodin and Rains uses the Eynard-Mehta theorem, which is a fact from linear algebra and probability. However, our proof uses the theory of symmetric functions, which arises from algebraic combinatorics. In this sense, our proof of \hyperref[process]{Theorem \ref*{process}} is the first to put the derivation of the correlation functions $\rho_{\textbf{S}}$ in a combinatorial setting. One might find this particularly appealing because many applications of the Schur measure and Schur process are to combinatorial questions.

The remainder of this article is organized as follows. In Section 2.1, we will recall several facts about symmetric functions; in Section 2.2, we will derive contour integral expressions for observables for the Schur measure; in Section 2.3, we will use this to obtain the correlation functions of the Schur measure; and in Section 3, we will generalize by evaluating the correlation functions of the Schur process. 

\section{Correlation Functions of the Schur Measure}

\subsection{Schur Polynomials and Scalar Products} 

In this section, we will state several facts about symmetric functions that will be used later in the article; many of these results can also be found in Macdonald's text \cite{22}. 

Suppose that $X = (x_1, x_2, \ldots , x_n)$ is a finite set of complex variables. Define the {\itshape power sums} by setting $p_0 (X) = 1$ and $p_k (X) = \sum_{i = 1}^n x_i^k$ for each positive integer $k$. For each partition $\lambda = (\lambda_1, \lambda_2, \ldots )$, set $p_{\lambda} (X) = \prod_{i = 1}^{\infty} p_{\lambda_i} (X)$.

Now let $Z = (z_1, z_2, \ldots , z_k)$ be a finite set of complex variables. Let $\Lambda (Z)$ denote the ring of symmetric polynomials in $Z$; equivalently, the elements of $Z$ are polynomial functions from $\mathbb{C}^n$ to $\mathbb{C}$ that are invariant under permutations of their arguments. Consider the ``truncated" bilinear form on $\Lambda(Z)$, which is fixed by setting 
\begin{flalign}
\label{scalarproduct}
\langle p_{\lambda} (Z), p_{\mu} (Z) \rangle_Z = \textbf{1}_{\lambda = \mu} \textbf{1}_{|\lambda| \le k} \displaystyle\prod_{i = 1}^{\infty} i^{m_{i} (\lambda)} (m_{i} (\lambda))!
\end{flalign}

\noindent for each $\lambda, \mu \in \mathbb{Y}$. 

This bilinear form has been used and discussed in \cite{22} in a slightly different setting, in which $Z$ is an infinite set of formal variables. However, many facts about the bilinear form that hold for infinite $k$ also have analogues for finite $k$. For instance, when $k = \infty$, the Schur functions $s_{\lambda} (Z)$ form an orthonormal basis of $\Lambda (Z)$ under this bilinear form (see Chapter 1, Section 4 of \cite{22}). Using this and the fact that the power sums $\{ p_{\lambda} (Z) \}_{\lambda \in \mathbb{Y}_n}$ span the space of degree $n$ elements of $\Lambda (Z)$ for each nonnegative integer $n$, we obtain the following result. 

\begin{prop}
Suppose that $Z = (z_1, z_2, \ldots , z_k)$ is a finite set of complex variables. Then, $\langle s_{\lambda} (Z), s_{\mu} (Z) \rangle_Z = \textbf{1}_{\lambda = \mu} \textbf{1}_{|\lambda| \le k}$. 
\end{prop}

From this, we deduce the following corollary, which can be viewed as a definition for the skew-Schur functions different from the one given in Chapter 1, Section 5 of \cite{22}. In the below, $s_{\lambda} (X, Z)$ denotes the Schur polynomial associated with $\lambda$ in the union of the variables given by $X \cup Z$. 

\begin{cor}
\label{skew}

Suppose that $X = (x_1, x_2, \ldots , x_n)$ and $Z = (z_1, z_2, \ldots , z_k)$ are finite sets of complex variables, and suppose that $\lambda, \mu \in \mathbb{Y}$. Then $\textbf{1}_{|\mu| \le k} s_{\lambda / \mu} (X) = \langle s_{\lambda} (X, Z), s_{\mu} (Z) \rangle_Z$. 
\end{cor}

Until now, we required that the arguments of the bilinear form be elements of $\Lambda (Z)$ and thus polynomials; however, this assumption may be weakened. Specifically, suppose that $a(Z) = \sum_{\lambda \in \mathbb{Y}} a_{\lambda} p_{\lambda} (Z)$ and $b (Z) = \sum_{\lambda \in \mathbb{Y}} b_{\lambda} p_{\lambda} (Z)$ are convergent power series in $Z$; since $a$ and $b$ are not necessarily polynomials, they are not necessarily elements of $\Lambda (Z)$. Consider the ``truncations" $a^{(k)} (Z) = \sum_{|\lambda| \le k} a_{\lambda} p_{\lambda} (Z)$ and $b^{(k)} (Z) = \sum_{|\lambda| \le k} b_{\lambda} p_{\lambda} (Z)$; both $a^{(k)} (Z)$ and $b^{(k)} (Z)$ are elements of $\Lambda (Z)$. We may define the scalar product $\langle a(Z), b(Z) \rangle_Z$ to be equal to $\langle a^{(k)} (Z), b^{(k)} (Z) \rangle_Z$. The resulting scalar product is still linear on the space of convergent power series in $Z$ due to the $\textbf{1}_{|\lambda| \le k}$ term in (\hyperref[scalarproduct]{\ref*{scalarproduct}}). 

Now, in the proof of \hyperref[process]{Theorem \ref*{process}}, we will put the skew-Schur polynomials appearing in the weight functions $\mathcal{W}_{X, Y}$ in terms of scalar products of Schur functions using \hyperref[skew]{Corollary \ref*{skew}} (see \hyperref[weightscalar]{Lemma \ref*{weightscalar}}). After performing several operations on the resulting expression (see \hyperref[weightoperators]{Proposition \ref*{weightoperators}} and \hyperref[contoursprocess]{Proposition \ref*{contoursprocess}}), we will obtain scalar products between certain types of rational functions. 

Specifically, recall that the {\itshape Cauchy product} is defined by 
\begin{flalign}
\label{equalityf}
F(X; Y) = \displaystyle\prod_{(x, y) \in X \times Y} (1 - xy)^{-1} = \exp \left( \displaystyle\sum_{j = 1}^{\infty} \displaystyle\frac{p_j (X) p_j (Y)}{j} \right), 
\end{flalign}

\noindent where the second equality (\hyperref[equalityf]{\ref*{equalityf}}) holds when the exponential converges (see Chapter 6, Section 2 of \cite{22}). For any parameter $q \in \mathbb{C}$, also define the product 
\begin{flalign}
\label{equalityh} 
H_q (X; Y) = \displaystyle\frac{F(X; Y)}{F(qX; Y)} = \exp \left( \displaystyle\sum_{j = 1}^{\infty} \displaystyle\frac{p_j (X) p_j (Y) (1 - q^j)}{j} \right),
\end{flalign}

\noindent where $qX = ( qx_1, qx_2, \ldots , qx_n)$ and the second equality (\hyperref[equalityh]{\ref*{equalityh}}) holds when the exponential converges (this follows from (\hyperref[equalityf]{\ref*{equalityf}})). 

The scalar products appearing in the proof of \hyperref[process]{Theorem \ref*{process}} will be between products of $F(X; Y)$ and $H_q (X; Y)$ (see \hyperref[fhscalar]{Lemma \ref*{fhscalar}}). Next, we will require a way to evaluate these scalar products. This will be done through the following lemma, which is similar to Proposition 2.3 of \cite{6}. 

\begin{lem}
\label{scalarseries}

Let $X$ be a finite set of complex variables, and let $Z = (z_1, z_2, \ldots )$ be an infinite set of complex variables. Let $Z_{[1, u]} = (z_1, z_2, \ldots , z_u)$ for each positive integer $u$, and suppose that $q_1, q_2, \ldots $ and $r_1, r_2, \ldots $ are power series in $X$. For each integer $u \ge 1$, let 
\begin{flalign*}
a \big( X, Z_{[1, u]} \big) &= \exp \left( \displaystyle\sum_{i=1}^\infty \displaystyle\frac{p_i \big(Z_{[1, u]} \big) q_i (X)}{i} \right); \quad b \big( X, Z_{[1, u]} \big) = \exp \left( \displaystyle\sum_{i=1}^\infty \displaystyle\frac{p_i \big(Z_{[1, u]} \big) r_i (X)}{i} \right); \\
c(X) &= \exp \left( \displaystyle\sum_{i=1}^\infty \displaystyle\frac{q_i (X) r_i (X)}{i} \right). 
\end{flalign*}

\noindent Suppose that $c(X)$ converges absolutely and that $a(X, Z_{[1, u]})$ and $b(X, Z_{[1, u]})$ converge absolutely for each positive integer $u$. Then,
\begin{flalign}
\label{scalarlimit}
\lim_{u \rightarrow \infty} \big\langle a \big(X, Z_{[1, u]} \big), b \big(X, Z_{[1, u]} \big) \big\rangle_{Z_{[1, u]}} = c(X).
\end{flalign}
\end{lem}

\begin{proof}
Observe that if $u_1, u_2, \ldots $ are power series in $Z$ and $v_1, v_2, \ldots $ are power series in $X$, then 
\begin{flalign}
\label{exponentiation}
\left( \displaystyle\sum_{i=1}^\infty \displaystyle\frac{u_i (Z) v_i (X)}{i} \right)^m = \displaystyle\sum_{\ell (\lambda) = m} \displaystyle\frac{m! \prod_{i = 1}^{\ell (\lambda)} u_{\lambda_i} (Z) v_{\lambda_i} (X)}{ \prod_{i = 1}^{\infty} i^{m_i (\lambda)} m_i (\lambda)!},
\end{flalign}

\noindent for any nonnegative integer $m$, where the sum is ranged over the partitions $\lambda = (\lambda_1, \lambda_2, \ldots )$ of length $m$. Applying (\hyperref[exponentiation]{\ref*{exponentiation}}) with $(u_i, v_i)$ equal to $(p_i, q_i)$, $(p_i, r_i)$, and $(q_i, r_i)$ and using the equality 
\begin{flalign}
\label{exponentialsum}
\exp \left( \displaystyle\sum_{i=1}^\infty \displaystyle\frac{u_i (Z) v_i (X)}{i} \right) = \displaystyle\sum_{j=0}^\infty \displaystyle\frac{1}{j!} \left( \displaystyle\sum_{i=1}^\infty \displaystyle\frac{u_i (Z) v_i (X)}{i} \right)^j
\end{flalign} 	

\noindent yields 
\begin{flalign}
\label{finitescalar}
\big\langle a \big( X, Z_{[1, u]} \big) , & b \big( X, Z_{[1, u]} \big) \big\rangle_{Z_{[1, u]}} \nonumber \\
&= \Bigg\langle \displaystyle\sum_{j = 0}^{\infty} \displaystyle\sum_{\ell (\lambda) = j} \displaystyle\frac{\prod_{i = 1}^{\ell (\lambda)} p_{\lambda_i} \big( Z_{[1, u]} \big) q_{\lambda_i} (X)}{\prod_{i = 1}^{\infty} i^{m_i (\lambda)} (m_i (\lambda))!} , \displaystyle\sum_{j = 0}^{\infty} \displaystyle\sum_{\ell (\lambda) = j} \displaystyle\frac{\prod_{i = 1}^{\ell (\lambda)} p_{\lambda_i} \big( Z_{[1, u]} \big) r_{\lambda_i} (X)}{\prod_{i = 1}^{\infty} i^{m_i (\lambda)} (m_i (\lambda))!} \Bigg\rangle_{Z_{[1, u]}} \nonumber \\
& = \displaystyle\sum_{j = 0}^{\infty} \displaystyle\sum_{\ell (\lambda) = j} \displaystyle\frac{\textbf{1}_{|\lambda| \le u} \prod_{i = 1}^{\ell (\lambda)} q_{\lambda_i} (X) r_{\lambda_i} (X)}{\prod_{i = 1}^{\infty} i^{m_i (\lambda)} (m_i (\lambda))!}
\end{flalign}

\noindent for all positive integers $u$, due to (\hyperref[scalarproduct]{\ref*{scalarproduct}}). Applying (\hyperref[exponentiation]{\ref*{exponentiation}}), (\hyperref[exponentialsum]{\ref*{exponentialsum}}), taking the limit as $u$ tends to $\infty$ in (\hyperref[finitescalar]{\ref*{finitescalar}}), and using absolute convergence of $c(X)$ then yields (\hyperref[scalarlimit]{\ref*{scalarlimit}}). 
\end{proof}

\subsection{Macdonald Difference Operators}
In this section, we will use the methods from Chapter 2 of \cite{5} to obtain contour integral expressions for a large class of observables for the Schur measure. 

Let $q \in \mathbb{C}$ be a parameter satisfying $|q| \in [0, 1)$, and let $X = (x_1, x_2, \ldots , x_n)$ be a finite set of complex variables. Let $T_{q, i}$ be the operator on $\Lambda (X)$ that sends any symmetric function $f(x_1, x_2, \ldots , x_n)\in \Lambda (X)$ to the symmetric function $f(x_1, x_2, \ldots, x_{i-1}, qx_i, x_{i+1}, \ldots , x_n) \in \Lambda (X)$. Define the {\itshape Macdonald $q$-difference operators} $D_{n; q}^r$ on $\Lambda (X)$ by 
 \begin{flalign*}
D_{n; q}^r = q^{\binom{r}{2}} \displaystyle\sum_{|I| = r} \displaystyle\prod_{\substack{i\in I \\ j \not\in I}} \displaystyle\frac{qx_i - x_j}{x_i - x_j} \displaystyle\prod_{i \in I}  T_{q, i},
 \end{flalign*}
 
\noindent where $I$ ranges over all subsets of $\{ 1, 2, \ldots , n \}$ of size $r$. The Macdonald $q$-difference operators are special $q = t$ cases of the Macdonald $(q, t)$-difference operators given in Chapter 6, Section 3 of \cite{22}. 

In \cite{5}, the variant of the Macdonald difference operator 
\begin{flalign*}
\tilde{D}_{n; q}^r = q^{- \binom{n}{2}} D_{n; q}^{n - r} T_{q^{-1}} 
\end{flalign*}

\noindent is defined, where $T_{q^{-1}} (F) (x_1, x_2, \ldots , x_n) = F(q^{-1} x_1, q^{-1} x_2, \ldots , q^{-1} x_n)$. We will use the operators $\tilde{D}_{n; q}^r$ instead of the operators $D_{n, q}^r$. 

The following result from Chapter 6, Section 4 of \cite{22} shows that the Schur polynomials are eigenfunctions of $\tilde{D}_{n; q}^r$. In the below, $e_r$ denotes the $r$th elementary symmetric polynomial. 

\begin{prop}
\label{eigenfunction}

Suppose $q \in \mathbb{C}$ satisfies $|q| \in [0, 1)$. For each $\lambda \in \mathbb{Y}$, the Schur polynomial $s_{\lambda} (X)$ is an eigenfunction of $\tilde{D}_{n; q}^r$ with eigenvalue $e_r (q^{1 - \lambda_1 - n}, q^{2 - \lambda_2 - n}, \ldots , q^{-\lambda_n})$. 
\end{prop}

\begin{rem}
For each $\lambda \in \mathbb{Y}$, the Schur polynomial $s_{\lambda} (X)$ is also an eigenfunction of $D_{n, q}^r$, but with eigenvalue $e_r (q^{n + \lambda_1 - 1}, q^{n + \lambda_2 - 2}, \ldots , q^{\lambda_n})$.  
\end{rem}

From Remark 2.2.11 of \cite{5}, we also have the following way of expressing the action of $\tilde{D}_{n; q}^1$ on certain types of functions $G$. 

\begin{prop}
\label{singleoperator}

Suppose that $q \in \mathbb{C}$ satisfies $|q| \in [0, 1)$, that $g$ is a rational function of one variable, and that $G$ is a function of $n$ variables satisfying $G(u_1, u_2, \ldots , u_n) = \prod_{i = 1}^n g(u_i)$. Further suppose that there exist positive numbers $1\le r < s$ such that for each $z \in \mathbb{C}$ with $|z| \in [r, s]$, we have that $g(z^{-1}) \ne 0$ and that $q^{-1} z^{-1}$ is not a pole of $g$. Then, for any $x_1, x_2, \ldots , x_n \in (s^{-1}, r^{-1})$, we have that  
\begin{flalign*}
q^n \tilde{D}_{n; q}^1 G(X) = \displaystyle\frac{G(X)}{2\pi i} \displaystyle\oint \displaystyle\frac{q}{z - zq} \left( \displaystyle\prod_{k = 1}^n \displaystyle\frac{1 - qz x_k}{1 - z x_k} \right) \left( \displaystyle\frac{g(q^{-1} z^{-1})}{g(z^{-1})} \right) dz,
\end{flalign*}

\noindent where the integral is along the union of the positively oriented circle $|z| = r$ and the negatively oriented circle $|z| = s$. 
\end{prop}

The below proposition generalizes \hyperref[singleoperator]{Proposition \ref*{singleoperator}}. Throughout, for any set of variables $X$ and any operator $D$ on $\Lambda (X)$, we will let $[D]_X$ denote the action of $D$ on $X$. 

\begin{prop}
\label{multipleoperators}

Suppose that $q_1, q_2, \ldots, q_m \in \mathbb{C}$ are complex numbers with magnitudes less than $1$, that $g$ is a rational function of one variable, and that $G$ is a function of $n$ variables satisfying $G(u_1, u_2, \ldots , u_n) = \prod_{i = 1}^n g(u_i)$. Let $1\le r_1, r_2, \ldots , r_m$ and $1\le s_1, s_2, \ldots , s_m$ be positive numbers sufficiently close to $1$ such that $\max_{1\le i\le m} |q_i| s_i < \min_{1 \le i \le m} r_i \le \max_{1\le i\le m} r_i < \min_{1\le i\le m} s_i$. Suppose that $g(z^{-1}) \ne 0$ and that $q_i^{-1} z^{-1}$ is not a pole of $g$ for each integer $i \in [1, m]$ and for each $z \in \mathbb{C}$ satisfying $\max_{1 \le i\le m} r_i \le |z| \le \min_{1\le i\le m} s_i$. Then, for any $x_1, x_2, \ldots , x_n \in (\max_{1\le i\le m} s_i^{-1}, \min_{1\le i\le m} r_i^{-1})$, we have that
\begin{flalign}
\label{operators}
\Big( \displaystyle\prod_{j = 1}^m q_j^n \tilde{D}_{n; q_j}^1 \Big) G(X) &= \displaystyle\frac{G(X)}{(2\pi i)^m} \displaystyle\oint \cdots \displaystyle\oint \displaystyle\prod_{j = 1}^m \displaystyle\frac{q_j}{z_j - q_j z_j} \displaystyle\prod_{1\le j < k\le m} \displaystyle\frac{(q_k z_k - q_j z_j)(z_k - z_j)}{(z_k - q_j z_j)(q_k z_k - z_j)} \nonumber \\
& \qquad \qquad \qquad \quad \times \displaystyle\prod_{j = 1}^m \left( \displaystyle\prod_{k = 1}^n \displaystyle\frac{1 - q_j z_j x_k}{1 - z_j x_k} \right) \displaystyle\frac{g(q_j^{-1} z_j^{-1})}{g(z_j^{-1})} \displaystyle\prod_{j=1}^m dz_j,
\end{flalign}

\noindent where the contour for $z_j$ is along the union of the positively oriented circle $|z_j| = r_j$ and the negatively oriented circle $|z_j| = s_j$ for each integer $j \in [1, m]$. 
\end{prop}

\begin{proof}
We will induct on $m$. If $m = 1$, then \hyperref[multipleoperators]{Proposition \ref*{multipleoperators}} coincides with \hyperref[singleoperator]{Proposition \ref*{singleoperator}}, so let us suppose that $m > 1$. Due to the inductive hypothesis and linearity of the $\tilde{D}_{n; q_i}$, we obtain that 
\begin{flalign*}
\Big( \displaystyle\prod_{j = 1}^{m} q_j^n \tilde{D}_{n; q_j}^1 \Big) G(X) &= q_{m}^n \tilde{D}_{n; q_m}^1 \Big( \displaystyle\prod_{j = 1}^{m - 1}  q_j^n \tilde{D}_{n; q_j}^1 \Big) G (X) \\
&= \displaystyle\frac{1}{(2\pi i)^{m - 1}} \displaystyle\oint \cdots \displaystyle\oint \displaystyle\prod_{j = 1}^{m - 1} \displaystyle\frac{q_j}{z_j - q_j z_j} \displaystyle\prod_{1\le j < k \le m - 1} \displaystyle\frac{(q_k z_k - q_j z_j)(z_k - z_j)}{(z_k - q_j z_j)(q_k z_k - z_j)}  \\
&\quad \times q_m^n [\tilde{D}_{n; q_m}^1]_X \left( G(X) \displaystyle\prod_{j = 1}^{m - 1} \left( \displaystyle\prod_{k = 1}^n \displaystyle\frac{1 - q_j z_j x_k}{1 - z_j x_k} \right) \displaystyle\frac{g(q_j^{-1} z_j^{-1})}{g(z_j^{-1})} \right) \displaystyle\prod_{j = 1}^{m - 1} dz_j, 
\end{flalign*}

\noindent where the contour for $z_j$ is the union of the positively oriented circle $|z_j| = r_j$ and the negatively oriented circle $|z_j| = s_j$ for each integer $j \in [1, m - 1]$. Now set 
\begin{flalign*}
g_1 (x) = g(x) \displaystyle\prod_{j = 1}^{m - 1} \displaystyle\frac{1 - q_j z_j x}{1 - z_j x}.
\end{flalign*}

\noindent For any $z \in \mathbb{C}$ satisfying $|z| \in [r_m, s_m]$, observe that $g_1 (z^{-1}) \ne 0$ and that $q_m^{-1} z^{-1}$ is not a pole of $g_1$ due to the conditions set on the $r_i$ and $s_i$. Then, applying \hyperref[singleoperator]{Proposition \ref*{singleoperator}} to $g_1$ instead of $g$ yields (\hyperref[operators]{\ref*{operators}}). 
\end{proof}

\begin{rem}
\label{determinants}
Already we begin to see determinantal expressions from \hyperref[multipleoperators]{Proposition \ref*{multipleoperators}}. Indeed, the first and second product of the right side of (\hyperref[operators]{\ref*{operators}}) yield a determinant due to the {\itshape Cauchy determinant identity} 
\begin{flalign}
\label{determinant}
\det \left[ \displaystyle\frac{1}{a_i - b_j} \right]_{i, j = 1}^n = \displaystyle\prod_{k = 1}^n \displaystyle\frac{1}{a_k - b_k} \displaystyle\prod_{1 \le j < k \le n} \displaystyle\frac{(a_k - a_j)(b_k - b_j)}{(a_k - b_j)(b_k - a_j)}, 
\end{flalign}

\noindent which holds for all sets of variables $(a_1, a_2, \ldots , a_n)$ and $(b_1, b_2, \ldots , b_n)$. This determinant arises in (\hyperref[operators]{\ref*{operators}}) because we are applying Macdonald $(q, q)$-operators instead of arbitrary Macdonald $(q, t)$-operators. If we had applied the more general operators, then the second product on the right side of (\hyperref[operators]{\ref*{operators}}) would depend on both the $q$ and $t$ parameters, and we would not immediately obtain a determinant. 
\end{rem}

\noindent Now, define the function 
\begin{flalign}
\label{generatingfunctionmeasure}
C (X; Y; Q) = F(X; Y)^{-1} \displaystyle\sum_{\lambda \in \mathbb{Y}} s_{\lambda} (X) s_{\lambda} (Y) \displaystyle\prod_{i = 1}^m \displaystyle\sum_{j = 1}^n q_i^{j - \lambda_j} 
\end{flalign}

\noindent for any finite sets of real numbers $X$ and $Y$ and any set of complex numbers $Q = \{ q_1, q_2, \ldots , q_m \} \subset \mathbb{C}$, when the right side of (\hyperref[generatingfunctionmeasure]{\ref*{generatingfunctionmeasure}}) converges. The following lemma puts the left side of (\hyperref[operators]{\ref*{operators}}) in terms of $C (X; Y; Q)$ when $G(X) = F(X; Y)$. 

\begin{lem}
\label{operatorsf}
Let $X = (x_1, x_2, \ldots , x_n)$ and $Y = (y_1, y_2, \ldots , y_n)$ be sets of nonnegative numbers less than $1$, and let $Q = \{ q_1, q_2, \ldots , q_m \} \subset \mathbb{C}$ be a set of complex numbers such that $|q_j|^m \in (\max Y, 1)$ for each integer $j \in [1, m]$. Then, $C (X; Y; Q)$ converges and
\begin{flalign}
\label{operatorfunction}
\Big( \displaystyle\prod_{j = 1}^m q_j^n [\tilde{D}_{n; q_j}^1]_X \Big) F(X; Y) = F(X; Y) C (X; Y; Q). 
\end{flalign}
\end{lem}

\begin{proof}
First, let us show that $C (X; Y; Q)$ converges. Recall (from a combinatorial interpretation of the Schur polynomials) that 
\begin{flalign}
\label{inequality}
0 \le s_{\lambda} (X) \le \big( |\lambda| + 1 \big)^{\ell (\lambda)^2}; \qquad 0 \le s_{\lambda} (Y) \le \big( |\lambda| + 1 \big)^{\ell (\lambda)^2} (\max Y)^{|\lambda|} 
\end{flalign}

\noindent for all $\lambda \in \mathbb{Y}$. This yields 
\begin{flalign*}
|F(X; Y) C (X; Y; Q)| &< \displaystyle\sum_{\lambda \in \mathbb{Y}} s_{\lambda} (X) s_{\lambda} (Y) \displaystyle\prod_{i = 1}^m \displaystyle\sum_{j = 1}^{\infty} |q_i|^{j - \lambda_j} \\ 
& \le \displaystyle\sum_{\ell (\lambda) \le n} \big( |\lambda| + 1 \big)^{2 n^2} \left( \displaystyle\frac{\max Y}{\min_{1\le k \le m} |q_k|^m} \right)^{|\lambda|} \displaystyle\prod_{i = 1}^m \displaystyle\sum_{j = 1}^{\infty} |q_i|^j \\
& \le \displaystyle\prod_{i = 1}^m |q_i| (1 - |q_i|)^{-1} \displaystyle\sum_{j = 0}^{\infty} |\mathbb{Y}_j| (j + 1)^{2 n^2} \left( \displaystyle\frac{\max Y}{\min_{1\le k \le n} |q_k|^m} \right)^j, 
\end{flalign*} 

\noindent which is finite because $\max Y < \min_{1 \le k \le n} |q_k|^m$ and because there exists a constant $c$ such that $|\mathbb{Y}_j| \le c^{\sqrt{j}}$ for all nonnegative integers $j$. This verifies the convergence of $C (X; Y; Q)$. 

Now, from the Cauchy identity (\hyperref[sum]{\ref*{sum}}), linearity of the $\tilde{D}_{n; q_j}^1$, and \hyperref[eigenfunction]{Proposition \ref*{eigenfunction}}, we obtain that 
\begin{flalign*}
\Big( \displaystyle\prod_{j = 1}^m q_j^n [\tilde{D}_{n; q_j}^1]_X \Big) F(X; Y) &= \displaystyle\prod_{j=1}^m q_j^n [\tilde{D}_{n; q_j}^1]_X \displaystyle\sum_{\lambda \in \mathbb{Y}} s_{\lambda} (X) s_{\lambda} (Y) \nonumber \\
&= \displaystyle\sum_{\lambda \in \mathbb{Y}} s_{\lambda} (X) s_{\lambda} (Y)\displaystyle\prod_{i = 1}^m \displaystyle\sum_{j = 1}^n q_i^{j - \lambda_j}.  
\end{flalign*}

\noindent This establishes (\hyperref[operatorfunction]{\ref*{operatorfunction}}). 
\end{proof} 

We will now apply \hyperref[multipleoperators]{Proposition \ref*{multipleoperators}} and \hyperref[operatorsf]{Lemma \ref*{operatorsf}} to obtain a contour integral form for $C (X; Y; Q)$. The following proposition is similar to Proposition 3.8 in \cite{6}, except here we apply the Macdonald $(q, t)$-operators with varying values of $q = t$.

\begin{prop}
\label{operatorsmeasure}
Let $X = (x_1, x_2, \ldots , x_n)$ and $Y = (y_1, y_2, \ldots , y_n)$ be sets of positive real numbers less than $1$, and let $q_1, q_2, \ldots , q_m \in \mathbb{C}$ be complex numbers such that $\max Y < |q_k|^m < 1$ for each integer $k \in [1, m]$. Let $R > r \ge 1$ be real numbers such that $R^{-1} < x_1, x_2, \ldots , x_n < r^{-1}$. Then 
\begin{flalign}
\label{operatormeasure} 
C (X; Y; Q) = \displaystyle\frac{1}{(2\pi i)^m} \displaystyle\oint \cdots & \displaystyle\oint \displaystyle\prod_{j = 1}^m \displaystyle\frac{q_j}{z_j - q_j z_j} \displaystyle\prod_{1\le j < k\le m} \displaystyle\frac{(q_k z_k - q_j z_j)(z_k - z_j)}{(z_k - q_j z_j)(q_k z_k - z_j)} \nonumber \\
& \times \displaystyle\prod_{j = 1}^m H_{q_j} (X; \{ z_j \}) H_{q_j} \left( Y; \{ q_j^{-1} z_j^{-1} \} \right) dz_j, 
\end{flalign}

\noindent where the contour for each $z_j$ is the union of the positively oriented circle $|z_j| = r$ and the negatively oriented circle $|z_j| = R$. 
\end{prop}

\begin{proof}

This proposition would follow from \hyperref[multipleoperators]{Proposition \ref*{multipleoperators}}, applied when $G(X) = F(X; Y)$, except that the contours might not coincide. In fact, for arbitrary sets of positive numbers $X$ and $Y$, contour radii $r_i$ and $s_i$ satisfying the conditions of \hyperref[multipleoperators]{Proposition \ref*{multipleoperators}} might not exist. Therefore, let us first assume that $x_1, x_2, \ldots , x_n \in (R^{-1}, r^{-1})$ are sufficiently close to $r^{-1}$ such that there exists a positive real number $R' \le R$ satisfying $R' \max_{1\le k \le m} |q_k| < r < \min (X^{-1}) \le \max (X^{-1}) < R'$. 

Applying \hyperref[multipleoperators]{Proposition \ref*{multipleoperators}} (with $G(X) = F(X; Y)$) to the left side of (\hyperref[operatorfunction]{\ref*{operatorfunction}}) yields (\hyperref[operatormeasure]{\ref*{operatormeasure}}), where the contour for each $z_k$ is now the union of the positively oriented circle $|z_k| = r$ and the negatively oriented circle $|z_k| = R'$. Therefore, it will suffice to show that we can deform each negatively oriented outer contour of radius $R'$ to a negatively oriented circle of radius $R$ without changing the value of the integral. 

To do so, we will use a method similar to the one applied in Proposition 3.8 of \cite{6}. Recall that the poles of the integrand contained in the original contours (of radii $r$ and $R'$) were all of the form $(x_{i_1}^{-1}, x_{i_2}^{-1}, \ldots , x_{i_m}^{-1})$ for some integers $i_1, i_2, \ldots , i_m \in [1, n]$. In order to show that the value of the integral does not change under the contour deformation, we will show that poles not of this type that are contained in the new contours have zero residue. To facilitate this, we will assume that the $x_i$ are pairwise distinct and that
\begin{flalign}
\label{condition}
x_i \ne x_{i'} \prod_{k = 1}^m q_k^{j_k}
\end{flalign}

\noindent for all integers $i, i' \in [1, m]$ and integers $j_1, j_2, \ldots , j_m \in [-2, 2]$ not all equal to $0$. Indeed, we may make this assumption due to the continuity of equality (\hyperref[operatormeasure]{\ref*{operatormeasure}}) in the variables $x_1, x_2, \ldots , x_n$ and $q_1, q_2, \ldots , q_m$. 

Now let us find all poles, and associated residues, of the integrand that are contained in the new contours. First integrate over $z_1$. Poles arise when either $z_1 \in X^{-1}$; $z_1 = q_1^{-1} z_j$ for some integer $j \ne 1$; or $z_1 = q_j z_j$ for some integer $j \ne 1$. In general, integrating over $z_k$ for any integer $k \in [1, m]$ yields poles when $z_k \in X^{-1}$; $z_k = q_k^{-1} z_j$ for some $j \ne k$; or $z_k = q_j z_j$ for some integer $j \ne k$. In the first case, we call the index $k$ {\itshape independent}; in the later two cases, we say that $k$ is {\itshape related} to $j$. For each index $j$, there is a sequence $j = j_0, j_1, \ldots , j_i$ of pairwise distinct indices such that $j_k$ is related to $j_{k + 1}$ for each integer $k \in [0, i - 1]$ and such that $j_i$ is independent; we call $j$ {\itshape dependent} on $j_i$. 

For instance, suppose that $m = 3$. The triple $(z_1, z_2, z_3) = (q_1^{-1} x_1^{-1}, x_1^{-1}, q_3^{-1} q_1^{-1} x_1^{-1})$ is one pole of the integrand. In this case, $3$ is related to $1$, which is related to $2$, which is independent. Observe that $3$ is not independent due to the assumption (\hyperref[condition]{\ref*{condition}}). Therefore, this pole is a simple pole, which implies that it has zero residue due to the factor of $H_{q_1} (X; \{ z_1 \})$ appearing in the integrand. Another example of a pole is $(z_1, z_2, z_3) = (q_1^{-1} x_1^{-1}, x_1^{-1}, x_1^{-1})$; here, $1$ is related to both $2$ and $3$, which are independent. The residue of this pole is also $0$ due to the factor of $z_2 - z_3$ in the integrand. 

In general, we will show that a pole of the integrand $(z_1, z_2, \ldots , z_m)$ has zero residue unless each integer $i \in [1, m]$ is independent. First observe that, if there are distinct independent indices $i, i' \in [1, m]$ such that $z_i = z_{i'}$, then the pole has zero residue due to the factor of $z_i - z_{i'}$ in the integrand; this generalizes the fact that the second example above $(z_1, z_2, z_3) = (q_1^{-1}, x_1^{-1}, x_1^{-1})$ has zero residue. 

Now, consider a pole of the integrand $(z_1, z_2, \ldots , z_m)$ such that there are no two distinct independent indices $i$ and $i'$ such that $z_i = z_{i'}$. Then for each integer $j \in [1, m]$, there is only one index $j_i \in [1, m]$ such that $j$ is dependent on $j_i$, due to assumption (\hyperref[condition]{\ref*{condition}}); this implies that $(z_1, z_2, \ldots , z_m)$ is a simple pole. Now, suppose that there is some integer $k \in [1, m]$ such that $z_k$ is not independent; we will show that $(z_1, z_2, \ldots , z_m)$ has zero residue. Without loss of generality, suppose that $k$ is not independent but is related to some independent index $i$; let $z_i = x_t^{-1}$ for some integer $t \in [1, n]$. Then, $z_k$ is either equal to $q_k^{-1} x_t^{-1}$ or $q_i x_t^{-1}$. Since $x_t^{-1} < R'$, the pole $q_i x_t^{-1}$ has magnitude less than $R' |q_i| < r$ and is thus not in the contour of integration; therefore, $z_k = q_k^{-1} x_t^{-1}$. Hence, the simple pole $(z_1, z_2, \ldots , z_m)$ has zero residue due to the factor of $H_{q_k} (X; \{ z_k \})$ in the integrand. 

This yields that all poles of the integrand contained in the new contours that have nonzero residue are of the form $(x_{i_1}^{-1}, x_{i_2}^{-1}, \ldots , x_{i_m}^{-1})$ for some $i_1, i_2, \ldots , i_m \in [1, m]$. These are the same as the poles contained in the original contour, so the value of the integral does not change under the contour deformation; this establishes the proposition when the $x_i$ are sufficiently close to $r^{-1}$. 

Now, observe that both sides of the equality (\hyperref[operatormeasure]{\ref*{operatormeasure}}) are analytic functions in complex variables $x_1, x_2, \ldots , x_n$ in the connected region where $|x_i| \in (r^{-1}, R^{-1})$ for each integer $i \in [1, n]$. By the above, (\hyperref[operatormeasure]{\ref*{operatormeasure}}) holds when the $x_i$ are sufficiently close to $r^{-1}$. Thus, (\hyperref[operatormeasure]{\ref*{operatormeasure}}) holds in general by uniqueness of analytic continuation. 
\end{proof}

\subsection{Proof of Theorem 1.1.1} 

In this section, we will establish \hyperref[measure]{Theorem \ref*{measure}} through the following weaker result. 

\begin{thm}
\label{measure2}

Suppose that $T = \{ t_1, t_2, \ldots , t_d \} \subset \mathbb{Z}$ is a set of pairwise distinct integers and that $n$ is an integer greater than $\max \{ d, d - \min T \}$. Let $X = (x_1, x_2, \ldots , x_n)$ and $Y = (y_1, y_2, \ldots , y_n)$ be any sets of $n$ positive real numbers less than $1$. Then, $\rho_{\textbf{SM}} (T) = \det \textbf{L}_T$, where $\textbf{L}_T$ is the $d \times d$ matrix defined in \hyperref[measure]{Theorem \ref*{measure}}. 
\end{thm}

\noindent Let us first show that \hyperref[measure2]{Theorem \ref*{measure2}} implies \hyperref[measure]{Theorem \ref*{measure}}. 
 
\begin{proof}[Proof of Theorem 1.1.1 Assuming Theorem 2.3.1]
\hyperref[measure]{Theorem \ref*{measure}} and \hyperref[measure2]{Theorem \ref*{measure2}} are different in three ways. The first difference is that \hyperref[measure2]{Theorem \ref*{measure2}} stipulates that $T$ consists of distinct integers, while $T$ is arbitrary in \hyperref[measure]{Theorem \ref*{measure}}. This can be resolved by observing that $\rho_{\textbf{SM}} (T) = 0 = \det \textbf{L}_T$ if $T$ contains two equal elements. Indeed, $\rho_{\textbf{SM}} (T) = 0$ in this case because $T$ contains two equal elements, while $\mathcal{X}(T)$ does not; furthermore, $\det \textbf{L}_T = 0$ because $\textbf{L}_T$ contains two equal columns. 

The second difference is that \hyperref[measure2]{Theorem \ref*{measure2}} assumes that the elements of $X$ and $Y$ are nonzero, while some of the elements of $X$ and $Y$ may be equal to zero in \hyperref[measure]{Theorem \ref*{measure}}. This can be resolved by observing that $\det \textbf{L}_T$ and $\rho_{\textbf{SM}} (T)$ are continuous functions in $X$ and $Y$.

The third difference is that \hyperref[measure2]{Theorem \ref*{measure2}} assumes that the sizes of the variable sets $X$ and $Y$ are equal and sufficiently large with respect to the set $T$, while the sizes of $X$ and $Y$ are arbitrary finite numbers in \hyperref[measure]{Theorem \ref*{measure}}. This can be resolved by setting some of the $x_i$ and $y_j$ to zero, thereby effectively reducing the sizes of $X$ and $Y$ to any positive integers at most equal to $n$. 
\end{proof}

\noindent We are now reduced to proving \hyperref[measure2]{Theorem \ref*{measure2}}. For any set of complex variables $Q = \{ q_1, q_2, \ldots , q_d \}$ and subset $U = \{ u_1, u_2, \ldots , u_d \} \subset \mathbb{Z}$, define the product $Q^{-U} = \prod_{i = 1}^d q_i^{-u_i}$. The following lemma shows how to extract the correlation function $\rho_{\textbf{SM}} (T)$ from $C (X; Y; Q)$. 

\begin{lem}
\label{functionmeasure}

Let $T$, $n$, $X$, and $Y$ be as in \hyperref[measure2]{Theorem \ref*{measure2}}. For any set of complex numbers $Q = \{ q_1, q_2, \ldots , q_d \} \subset \mathbb{C}$ satisfying $\max Y < |q_i|^d < 1$ for each integer $i \in [1, d]$, the correlation function $\rho_{\textbf{SM}} (T)$ is equal to the coefficient of $Q^{-T}$ in $C (X; Y; Q)$. Equivalently, 
\begin{flalign*}
\rho_{\textbf{SM}} (T) = \displaystyle\frac{1}{(2 \pi i)^d} \displaystyle\oint \cdots \displaystyle\oint C(X; Y; Q) \displaystyle\prod_{j = 1}^d q_j^{t_j - 1} d q_j,
\end{flalign*}

\noindent where the contour for each $q_j$ is the positively oriented circle $|q_j| = r_j$, where $r_1, r_2, \ldots , r_d$ are any positve real numbers greater than $(\max Y)^{1 / d}$ and less than $1$. 
\end{lem}

\begin{proof}
By \hyperref[operatorsf]{Lemma \ref*{operatorsf}}, $C(X; Y; Q)$ is a convergent power series and thus an analytic function in $Q$ when $\max Y < |q_i|^d < 1$ for each integer $i \in [1, d]$. Therefore, the equivalence of the two statements of the lemma follows from the residue theorem. 

Let us verify the first statement of the lemma. For any convergent power series $P_1$ and $P_2$ in $Q$, we say that $P_1 \simeq P_2$ if the coefficients of $Q^{-T}$ in $P_1$ and $P_2$ are equal. We have that 
\begin{flalign*}
\displaystyle\sum_{U \in \mathbb{Z}^d} \rho_{\textbf{SM}} (U) Q^{-U} &= \displaystyle\sum_{\lambda \in \mathbb{Y}} \textbf{SM}(\lambda) \displaystyle\sum_{U\in \mathbb{Z}^d}\textbf{1}_{U \subseteq \mathcal{X}(\lambda)} Q^{-U} \\
& \simeq F(X; Y)^{-1} \displaystyle\sum_{\lambda \in \mathbb{Y}} s_{\lambda} (X) s_{\lambda} (Y) \displaystyle\prod_{i=1}^d \displaystyle\sum_{j=1}^n q_i^{j - \lambda_j} \\
& = C (X; Y; Q). 
\end{flalign*}

\noindent The first equality above is due to (\hyperref[correlationmeasure]{\ref*{correlationmeasure}}). The second equality holds because the coefficient of $Q^{-T}$ in $\prod_{i = 1}^d \sum_{j = 1}^n q_i^{j - \lambda_j}$ is $1$ if $T \subset \mathcal{X} (\lambda)$ and is $0$ otherwise; this is due to the fact that $\mathcal{X}(\lambda)$ consists of distinct elements greater than $-n$ if $\ell (\lambda) \le n$, and due to our assumptions that $T$ consists of pairwise distinct elements and that $n > \max \{ d, d - \min T \}$. The third equality is due to the definition (\hyperref[generatingfunctionmeasure]{\ref*{generatingfunctionmeasure}}). 
\end{proof}

\noindent The following proposition uses \hyperref[operatorsmeasure]{Proposition \ref*{operatorsmeasure}} and \hyperref[functionmeasure]{Lemma \ref*{functionmeasure}} to obtain a contour integral form for $\rho_{\textbf{SM}} (T)$. 

\begin{prop}
\label{contourmeasure}

Let $T$, $n$, $X$, and $Y$ be as in \hyperref[measure2]{Theorem \ref*{measure2}}. Then
\begin{flalign*}
\rho_{\textbf{SM}} (T) = \displaystyle\frac{1}{(-4\pi^2)^d} \displaystyle\oint \cdots \displaystyle\oint \det \left[ \displaystyle\frac{1}{z_j - q_k z_k} \right]_{j, k = 1}^d \displaystyle\prod_{j = 1}^d H_{q_j} (X; \{ z_j \}) H_{q_j} (Y; \{ q_j^{-1} z_j^{-1} \} ) q_j^{t_j} d q_j d z_j, 
\end{flalign*}

\noindent where the contour for each $z_j$ is the positively oriented circle $|z_j| = 1$ and the contour for each $q_j$ is the positively oriented circle $|q_j| = r_j$, for any $r_1, r_2, \ldots , r_d \in ((\max Y)^{1 / d}, 1)$. 
\end{prop}
 
\begin{proof}
Applying \hyperref[operatorsmeasure]{Proposition \ref*{operatorsmeasure}} and the Cauchy determinant identity (\hyperref[determinant]{\ref*{determinant}}) when $\{ a_1, a_2, \ldots , a_d \} = \{ z_1, z_2, \ldots , z_d \}$ and $\{ b_1, b_2, \ldots , b_d \} = \{ q_1 z_1, q_2 z_2, \ldots , q_d z_d \}$ yields $C(X; Y; Q)$ is equal to 
\begin{flalign}
\label{contourgeneratingfunctionmeasure}
\displaystyle\frac{1}{(2 \pi i)^d} \displaystyle\oint \cdots \displaystyle\oint \det \left[ \displaystyle\frac{1}{z_j - q_k z_k} \right]_{j, k = 1}^d \displaystyle\prod_{j = 1}^d H_{q_j} (X; \{ z_j \}) H_{q_j} (Y; \{ q_j^{-1} z_j^{-1} \}) q_j dz_j ,  
\end{flalign}

\noindent where the contour for each $z_j$ is the union of the positively oriented circle $|z_j| = 1$ and the negatively oriented circle $|z_j| = R$ for any $R > \max (X^{-1})$. This integral is equal to the sum of $2^m$ integrals, in which each variable $z_j$ is either integrated along a circle of radius $1$ or of radius $R$. Let $R$ tend to $\infty$, and consider any summand in which some variable, say $z_k$, is integrated along a circle of radius $R$. The factor of $H_{q_k} (X; \{ z_k \})$ in (\hyperref[contourgeneratingfunctionmeasure]{\ref*{contourgeneratingfunctionmeasure}}) tends to $(q_k)^n$, and the factor of $H_{q_k} (Y; \{ q_k^{-1} z_k^{-1} \})$ tends to $1$. Since the exponent of $q_k$ in $Q^{-T}$ is less than $n$ (due to our assumption $n > \max \{ d, d - \min T \}$), the coefficient of $Q^{-T}$ in any such summand tends to $0$ as $R$ tends to $\infty$. Hence, $\rho_{\textbf{SM}} (T)$ is the coefficient of $Q^{-T}$ in (\hyperref[contourgeneratingfunctionmeasure]{\ref*{contourgeneratingfunctionmeasure}}) in which each $z_j$ is integrated along the positively oriented circle $|z_j| = 1$. 

Therefore, $\rho_{\textbf{SM}} (T)$ is equal to the residue of the pole $(q_1, q_2, \ldots , q_d) = (0, 0, \ldots , 0)$ of the integral 
\begin{flalign}
\label{coefficientgeneratingfunction}
\displaystyle\frac{1}{(2\pi i)^{d}} \displaystyle\oint \cdots \displaystyle\oint \det \left[ \displaystyle\frac{1}{z_j - q_k z_k} \right]_{j, k = 1}^d \displaystyle\prod_{j = 1}^d H_{q_j} (X; \{ z_j \}) H_{q_j} (Y; \{ q_j^{-1} z_j^{-1} \}) q_j^{t_j} d z_j, 
\end{flalign} 

\noindent where the contour for each $z_j$ is the positively oriented circle $|z_j| = 1$. Due to the residue theorem, $\rho_{\textbf{SM}} (T)$ is obtained from integrating each $q_j$ in (\hyperref[coefficientgeneratingfunction]{\ref*{coefficientgeneratingfunction}}) along a circle of radius $r_j$ for each integer $j \in [1, d]$. This yields the proposition. 
\end{proof}

\noindent We may now establish \hyperref[measure2]{Theorem \ref*{measure2}}. 

\begin{proof}[Proof of Theorem 2.3.1] 

Let $S_d$ be the symmetric group on $d$ elements. Applying \hyperref[contourmeasure]{Proposition \ref*{contourmeasure}} and expanding the determinant as a signed sum, we obtain 
\begin{flalign*}
\rho_{\textbf{SM}} (T) &= \displaystyle\frac{1}{(-4 \pi^2)^d} \displaystyle\oint \cdots \displaystyle\oint \displaystyle\sum_{\sigma \in S_d} \text{sgn} (\sigma) \displaystyle\prod_{j = 1}^d \displaystyle\frac{F(X; \{ z_j \}) F(Y; \{ q_j^{-1} z_j^{-1} \}) q_j^{t_j} d q_j d z_j}{\big( z_j - q_{\sigma (j)} z_{\sigma (j)} \big) F(X; \{ q_j z_j \}) F(Y; \{ z_j^{-1} \})}  \\
&= \displaystyle\frac{1}{(-4 \pi^2)^d} \displaystyle\sum_{\sigma \in S_d} \text{sgn} (\sigma) \displaystyle\oint \cdots \displaystyle\oint \displaystyle\prod_{j = 1}^d \displaystyle\frac{F(X; \{ z_j \}) F \big( Y; \{ q_{\sigma (j)}^{-1} z_{\sigma (j)}^{-1} \} \big) \big( q_{\sigma (j)} z_{\sigma (j)} \big)^{t_{\sigma (j)}} d q_j d z_j}{\big( z_j - q_{\sigma (j)} z_{\sigma (j)} \big) F \big( X; \{ q_{\sigma (j)} z_{\sigma (j)} \} \big) F(Y; \{ z_j^{-1} \}) z_j^{t_j}}, 
\end{flalign*} 

\noindent where each $z_j$ is integrated along the positively oriented circle $|z_j| = 1$ and each $q_j$ is oriented along the positively oriented circle $|q_j| = r$, for any real number $r \in ((\min Y)^{1 / d}, 1)$. Substituting $w_k = q_k z_k$ for each integer $k \in [1, d]$ yields 
\begin{flalign}
\label{determinantmeasure}
\rho_{\textbf{SM}} (T) &= \displaystyle\frac{1}{(-4 \pi^2)^d} \displaystyle\sum_{\sigma \in S_d} \text{sgn} (\sigma) \displaystyle\oint \cdots \displaystyle\oint \displaystyle\prod_{j = 1}^d \displaystyle\frac{F(X; \{ z_j \}) F \big( Y; \{ w_{\sigma (j)}^{-1} \} \big) w_{\sigma (j)}^{t_{\sigma (j)}} d w_{\sigma (j)} d z_j}{ \big( z_j - w_{\sigma (j)} \big) F \big( X; \{ w_{\sigma (j)} \} \big) F(Y; \{ z_j^{-1} \}) z_j^{t_j + 1}} \nonumber \\
&= \displaystyle\sum_{\sigma \in S_d} \text{sgn} (\sigma) \displaystyle\prod_{j = 1}^d \displaystyle\frac{1}{4 \pi^2} \displaystyle\oint \displaystyle\oint \displaystyle\frac{F(X; \{ z_j \}) F \big(Y; \{ w_{\sigma (j)}^{-1} \} \big) w_{\sigma (j)}^{t_{\sigma (j)}} d w_{\sigma (j)} d z_j}{ \big(  w_{\sigma (j)} - z_j \big) F \big( X; \{ w_{\sigma (j)} \} \big) F(Y; \{ z_j^{-1} \}) z_j^{t_j + 1}}, 
\end{flalign}

\noindent where the contour for each $z_j$ is the positively oriented circle $|z_j| = 1$ and the contour for each $q_j$ is the positively oriented circle $|q_j| = r$. The right side of (\hyperref[determinantmeasure]{\ref*{determinantmeasure}}) is the signed sum expansion of the determinant of $\textbf{L}_T$, whose entries are given by (\hyperref[kernelmeasure]{\ref*{kernelmeasure}}) but whose contours are possibly different. Specifically, the contour for $z$ is now the positively oriented circle $|z| = 1$ and the contour for $w$ is now the positively oriented circle $|w| = r$. However, we may deform the outer contour to $|z| = r_1$ and the inner contour to $|w| = r_2$ for any $\max X, \max Y < r_2 < r_1 < \min (X^{-1}), \min (Y^{-1})$ without changing the value of the expression (\hyperref[determinantmeasure]{\ref*{determinantmeasure}}), because this contour deformation does not pass through any poles of the integrands. Hence, we deduce that $\rho_{\textbf{SM}} (T) = \det \textbf{L}_T$. 
\end{proof}

\section{Correlation Functions of the Schur Process}

The proofs of \hyperref[process]{Theorem \ref*{process}} and \hyperref[measure]{Theorem \ref*{measure}} will be different in three ways. First, in order to establish \hyperref[process]{Theorem \ref*{process}}, we must use a generalized version of the generating function $C(X; Y; Q)$; this will be discussed in Section 3.1. Second, we will use the results from Section 2.1 to put this generating function in terms of a nested scalar product; this will be done in Section 3.2. Third, we will use the results of Section 2.1 and Section 2.2 to evaluate this scalar product and derive the correlation functions of the Schur process; this will be done in Section 3.3. 

\subsection{A Generalized Generating Function} 

Similar to \hyperref[measure]{Theorem \ref*{measure}}, we will deduce \hyperref[process]{Theorem \ref*{process}} from a weaker result. 

\begin{thm}
\label{process2}

For each integer $i \in [1, m]$, let $T_i = \{ t_{i, 1}, t_{i, 2}, \ldots , t_{i, d_i} \} \subset \mathbb{Z}$ be a (possibly empty) finite set of pairwise distinct integers. Let $T = \bigcup_{i = 1}^m \bigcup_{j = 1}^{d_i} \{ (i, t_{i, j}) \} \subset \{ 1, 2, \ldots , m \} \times \mathbb{Z}$, and let $|T| = \sum_{i = 1}^m d_i = d$. Let $n$ be some integer greater than $\max \{ d, d - \max \bigcup_{i = 1}^m T_i \}$, and let $X^{(1)}, X^{(2)}, \ldots , X^{(m)}$ and $Y^{(1)}, Y^{(2)}, \ldots , Y^{(m)}$ each be sets of $n$ positive real numbers less than $1$. Then, $\rho_{\textbf{S}} (T) = \det \textbf{K}_T$, where $\textbf{K}_T$ is the $d \times d$ matrix defined in \hyperref[process]{Theorem \ref*{process}}. 
\end{thm}

The proof of \hyperref[process]{Theorem \ref*{process}} assuming \hyperref[process2]{Theorem \ref*{process2}} is similar to the proof of \hyperref[measure]{Theorem \ref*{measure}} assuming \hyperref[measure2]{Theorem \ref*{measure2}}, so we omit it. For the remainder of this section, we will suppose that the sets $T_1, T_2, \ldots , T_m$, the integer $n$, and the sets of numbers $X^{(1)}, X^{(2)}, \ldots , X^{(m)}, Y^{(1)}, Y^{(2)}, \ldots , Y^{(m)}$ satisfy the conditions of \hyperref[process2]{Theorem \ref*{process2}}. 

Now, order to establish \hyperref[process2]{Theorem \ref*{process2}}, we will define a generating function that generalizes the function $C(X; Y; Q)$ from Section 2. For any sets $X^{(1)}, X^{(2)}, \ldots , X^{(m)}, Y^{(1)}, Y^{(2)}, \ldots , Y^{(m)}$ of positive numbers less than $1$; any set of complex numbers $Q = \{ q_{i, j} \} \subset \mathbb{C}$, where $i$ ranges from $1$ to $m$ and $j$ ranges from $1$ to $d_i$; and any (possibly infinite) integer $u \ge 0$, define the function 
\begin{flalign}
\label{generatingfunctionprocess}
C(X; Y; Q; u) = \displaystyle\sum_{(\lambda, \mu) \in \mathbb{Y}^m \times \mathbb{Y}^{m - 1}} \textbf{S} (\lambda, \mu) \displaystyle\prod_{j = 1}^{d_1} \displaystyle\sum_{k = 1}^n q_{1, j}^{k - \lambda_k^{(1)}} \displaystyle\prod_{i = 2}^m \textbf{1}_{|\mu^{(i - 1)}| \le u} \displaystyle\prod_{j = 1}^{d_i} \displaystyle\sum_{k = 1}^{n + u} q_{i, j}^{k - \lambda_k^{(i)}}, 
\end{flalign}

\noindent where $\lambda = (\lambda^{(1)}, \lambda^{(2)}, \ldots , \lambda^{(m)})$ and $\mu = (\mu^{(1)}, \mu^{(2)}, \ldots , \mu^{(m - 1)})$, when it converges. The following lemma gives a sufficient condition for convergence of $C (X; Y; Q; u)$. 

\begin{lem}
\label{generatingfunctioninequality}

If $\max \bigcup_{i = 1}^m Y^{(i)} < |q|^{dm^2} < 1$ for each $q \in Q$, then $C(X; Y; Q; u)$ converges absolutely for each (possibly infinite) nonnegative integer $u$. 
\end{lem}

\begin{proof}
The proof is similar to the proof of convergence in \hyperref[operatorsf]{Lemma \ref*{operatorsf}}. In analogue with (\hyperref[inequality]{\ref*{inequality}}), we have that 
\begin{flalign}
\label{skewinequality}
0 \le s_{\kappa / \nu} (X^{(i)}) \le \big( |\kappa| + 1 \big)^{\ell (\kappa)^2}; \quad 0 \le s_{\kappa / \nu} (Y^{(i)}) \le \big( |\kappa| + 1 \big)^{\ell (\kappa)^2} \big( \max Y^{(i)} \big)^{|\kappa| - |\nu|},  
\end{flalign}

\noindent for all partitions $\kappa, \nu \in \mathbb{Y}$ and all integers $i \in [1, m]$, due to a combinatorial interpretation of the skew-Schur polynomials. For any positive integer $i$ and any $\lambda = \{ \lambda^{(1)}, \lambda^{(2)}, \ldots , \lambda^{(i)} \} \in \mathbb{Y}^i$, define $|\lambda| = \sum_{j = 1}^i |\lambda^{(j)}|$. Applying (\hyperref[skewinequality]{\ref*{skewinequality}}), we obtain 
\begin{flalign}
\label{weightinequality}
\mathcal{W}_{X, Y} (\lambda, \mu) \le \big( \max \cup_{i = 1}^m Y^{(i)} \big)^{|\lambda| - |\mu|} \displaystyle\prod_{i = 1}^m \big( |\lambda^{(i)}| + 1 \big)^{2n^2}
\end{flalign}

\noindent for all $(\lambda, \mu) \in \mathbb{Y}^m \times \mathbb{Y}^{m - 1}$. 

Now, let $h \in [1, m]$ be the integer such that $|\lambda^{(h)}|$ is maximum. Recall that $\mathcal{W}_{X, Y} (\lambda, \mu)$ is $0$ unless $\mu^{(i)} \subseteq \lambda^{(i)}$ and $\mu^{(i)} \subseteq \lambda^{(i + 1)}$ for each integer $i \in [1, m - 1]$; these yield that $|\mu^{(i)}| \le |\lambda^{(i)}|$ and $|\mu^{(i)}| \le |\lambda^{(i + 1)}|$ for each integer $i \in [1, m - 1]$. Applying the first inequality for all integers $i \in [1, h - 1]$, applying the second inequality for all integers $i \in [h, m]$, and summing yields that either $\mathcal{W}_{X, Y} (\lambda, \mu) = 0$ or $|\lambda| - |\mu| \ge |\lambda^{(h)}| \ge |\lambda| / m$. 

Inserting this into (\hyperref[weightinequality]{\ref*{weightinequality}}) and recalling that $Z_{X, Y} = \prod_{1\le i\le j \le m} F \big( X^{(i)}; Y^{(j)} \big)$ yields 
\begin{flalign*}
Z_{X, Y} |C(X; Y; Q; u)| &\le \displaystyle\sum_{(\lambda, \mu)} \mathcal{W}_{X, Y} (\lambda, \mu) \displaystyle\prod_{i = 1}^m \displaystyle\sum_{j = 1}^{d_i} \displaystyle\sum_{k = 1}^{\infty} |q_{i, j}|^{k - \lambda_k^{(i)}} \\
&\le \displaystyle\sum_{(\lambda, \mu)} \big( \max \cup_{i = 1}^m Y^{(i)} \big)^{|\lambda| / m} \displaystyle\prod_{i = 1}^m \big( |\lambda^{(i)}| + 1 \big)^{2n^2} \displaystyle\prod_{j = 1}^{d_i} \displaystyle\sum_{k = 1}^{\infty} |q_{i, j}|^{k - \lambda_k^{(i)}} \\
&\le \displaystyle\sum_{(\lambda, \mu)} \left( \displaystyle\frac{\max \cup_{i = 1}^m Y^{(i)}}{\min_{q \in Q} |q|^{dm^2}} \right)^{|\lambda| / m} \displaystyle\prod_{i = 1}^m \big( |\lambda^{(i)}| + 1 \big)^{2n^2} \displaystyle\prod_{j = 1}^{d_i} \displaystyle\sum_{k = 1}^{\infty} |q_{i, j}|^k \\
&\le \displaystyle\sum_{(\lambda, \mu)} \left( \displaystyle\frac{\max \cup_{i = 1}^m Y^{(i)}}{\min_{q \in Q} |q|^{dm^2}} \right)^{|\lambda| / m} \left( \displaystyle\frac{|\lambda|}{m} + 1 \right)^{2mn^2} \displaystyle\prod_{i = 1}^m \displaystyle\prod_{j = 1}^{d_i} q_{i, j} (1 - q_{i, j})^{-1}, 
\end{flalign*}

\noindent where $\lambda = (\lambda^{(1)}, \lambda^{(2)}, \ldots , \lambda^{(m)}) \subset \mathbb{Y}^m$ and $\mu = (\mu^{(1)}, \mu^{(2)}, \ldots , \mu^{(m - 1)}) \in \mathbb{Y}^{m - 1}$ are summed over all sequences of partitions satisfying $\lambda^{(1)} \supseteq \mu^{(1)} \subseteq \lambda^{(2)} \supseteq \cdots \subseteq \lambda^{(m)}$. 

Now, there exists a constant $c > 0$ such that $\big| \bigcup_{k = 0}^j \mathbb{Y}_j \big| \le c^{\sqrt{j}}$ for each nonnegative integer $j$. Therefore, for any fixed $\lambda \in \mathbb{Y}^m$ such that $|\lambda| = j$, there are at most $c^{(m - 1) \sqrt{j}}$ sequences $\mu \in \mathbb{Y}^{m - 1}$ satisfying $\lambda^{(1)} \supseteq \mu^{(1)} \subseteq \lambda^{(2)} \supseteq \cdots \subseteq \lambda^{(m)}$. Furthermore, for each nonnegative integer $j$, there are at most $c^{m \sqrt{j}}$ sequences $\lambda \in \mathbb{Y}^m$ such that $|\lambda| = j$. Hence there is a constant $c'$ only dependent on $m$, $n$, $X^{(1)}, X^{(2)}, \ldots , X^{(m)}, Y^{(1)}, Y^{(2)}, \ldots , Y^{(m)}$, and $Q$ such that 
\begin{flalign*}
|C(X; Y; Q; u)| &\le c' \displaystyle\sum_{j = 0}^{\infty} \displaystyle\sum_{|\lambda| \le j} \displaystyle\sum_{|\mu| \le j} \left( \displaystyle\frac{\max \cup_{i = 1}^m Y^{(i)}}{\min_{q \in Q} |q|^{dm^2}} \right)^{j / m} \left( \displaystyle\frac{j}{m} + 1 \right)^{2mn^2} \\
&\le c' \displaystyle\sum_{j = 0}^{\infty} c^{(2m - 1) \sqrt{j}} \left( \displaystyle\frac{\max \cup_{i = 1}^m Y^{(i)}}{\min_{q \in Q} |q|^{dm^2}} \right)^{j / m} \left( \displaystyle\frac{j}{m} + 1 \right)^{2mn^2}, 
\end{flalign*}

\noindent which is finite because $\max \bigcup_{i = 1}^m Y^{(i)} < \min_{q \in Q} |q|^{dm^2}$. This establishes the lemma. 
\end{proof}

\noindent Due to \hyperref[generatingfunctioninequality]{Lemma \ref*{generatingfunctioninequality}}, we may express $C(X; Y; Q; \infty)$ as a limit of $C(X; Y; Q; u)$. 

\begin{cor}
\label{limit}

If $\max \bigcup_{i = 1}^m Y^{(i)} < |q|^{dm^2} < 1$ for each $q \in Q$, then we have that $C(X; Y; Q; \infty) = \lim_{u \rightarrow \infty} C(X; Y; Q; u)$. 
\end{cor}

\begin{proof}
This follows from the definition (\hyperref[generatingfunctionprocess]{\ref*{generatingfunctionprocess}}) and \hyperref[generatingfunctioninequality]{Lemma \ref*{generatingfunctioninequality}}. 
\end{proof}

\noindent For any set of complex variables $Q = \{ q_{i, j} \} \subset \mathbb{C}$, where $i$ ranges from $1$ to $m$ and $j$ ranges from $1$ to $d_i$, define the product $Q^{-T} = \prod_{i = 1}^m \prod_{j = 1}^{d_i} q_{i, j}^{-t_{i, j}}$. In analogue with \hyperref[functionmeasure]{Lemma \ref*{functionmeasure}}, we have the following lemma that shows how to obtain $\rho_{\textbf{S}} (T)$ from $C(X; Y; Q; \infty)$. 

\begin{lem}
\label{functionprocess}

If $Q = \{ q_{i, j} \}$ is a set of complex variables with magnitudes less than $1$ such that such that $\max \bigcup_{i = 1}^m Y^{(i)} < \min_{q \in Q} |q|^{dm^2}$, then $\rho_{\textbf{S}} (T)$ is the coefficient of $Q^{-T}$ in $C(X; Y; Q; \infty)$. Equivalently,
\begin{flalign*}
\rho_{\textbf{S}} (T) = \displaystyle\frac{1}{(2 \pi i)^d} \displaystyle\oint \cdots \displaystyle\oint C(X; Y; Q; \infty) \displaystyle\prod_{i = 1}^m \displaystyle\prod_{j = 1}^{d_i} q_{i, j}^{t_{i, j} - 1} d q_{i, j},
\end{flalign*}

\noindent where the contour for each $q_{i, j}$ is the positively oriented circle $|q_{i, j}| = r_{i, j}$, where $r_{i, j}$ are arbitrary positive numbers satisfying $\max \bigcup_{i = 1}^m Y^{(i)} < (r_{i, j})^{dm^2} < 1$ for each $i$ and $j$. 
\end{lem} 

We omit the proof of \hyperref[functionprocess]{Lemma \ref*{functionprocess}} because it is similar to that of \hyperref[functionmeasure]{Lemma \ref*{functionmeasure}}. 

\subsection{A Nested Scalar Product} 

\label{ProductW}

Now, we wish to find an analogue of (\hyperref[operatorfunction]{\ref*{operatorfunction}}) for $C(X; Y; Q; \infty)$. The proof of (\hyperref[operatorfunction]{\ref*{operatorfunction}}) used the fact that the Schur polynomials are eigenfunctions of the Macdonald $q$-difference operators. However, the function $C(X; Y; Q; u)$ is expressed in terms of the Schur process weights $\mathcal{W}_{X, Y} (\lambda, \mu)$; these contain products of skew-Schur functions, which are not always eigenfunctions of the Macdonald difference operators. Therefore, in order to apply \hyperref[eigenfunction]{Proposition \ref*{eigenfunction}}, we will first use \hyperref[skew]{Corollary \ref*{skew}} to express the weights $\mathcal{W}_{X, Y} (\lambda, \mu)$ in terms of scalar products of Schur functions. 

\begin{lem}
\label{weightscalar}

Let $A^{(1)}, A^{(2)}, \ldots , A^{(m - 1)}$ and $B^{(1)}, B^{(2)}, \ldots , B^{(m - 1)}$ be countably infinite sets of complex variables. For all integers $i \in [1, m - 1]$ and $u \ge 1$, let $A_{[1, u]}^{(i)}$ denote the finite set consisting of the first $u$ elements of $A^{(i)}$; define $B_{[1, u]}^{(i)}$ similarly. For each positive integer $u$, we have that 
\begin{flalign*}
\mathcal{W}_{X, Y} (\lambda, \mu) \displaystyle\prod_{i = 1}^{m - 1} \textbf{1}_{|\mu^{(i)}| \le u} = s_{\lambda^{(1)}} \big(X^{(1)} \big) & \displaystyle\prod_{i = 1}^{m-1} \Bigg( \Big\langle s_{\lambda^{(i+1)}} \big(X^{(i + 1)}, A_{[1, u]}^{(i)} \big), s_{\mu^{(i)}} \big(A_{[1, u]}^{(i)} \big) \Big\rangle_{A_{[1, u]}^{(i)}} \\
& \quad \times \Big\langle s_{\lambda^{(i)}} \big(Y^{(i)}, B_{[1, u]}^{(i)} \big), s_{\mu^{(i)}} \big(B_{[1, u]}^{(i)} \big) \Big\rangle_{B_{[1, u]}^{(i)}} \Bigg) s_{\lambda^{(m)}} \big(Y^{(m)} \big). 
\end{flalign*}
\end{lem}

\begin{proof}
This follows from applying \hyperref[skew]{Corollary \ref*{skew}} to (\hyperref[weight]{\ref*{weight}}). 
\end{proof}

\noindent Now, in addition to the assumptions on $T$, $n$, $X^{(1)}, X^{(2)}, \ldots , X^{(m)}$, and $Y^{(1)}, Y^{(2)}, \ldots , Y^{(m)}$ made above, we will also assume the following for the remainder of this section. 
\begin{itemize}
\item{Let $s_1, s_2, \ldots , s_m$ be positive numbers (which will be contour radii) less than $1$ and greater than $\big( \max \bigcup_{i = 1}^m Y^{(i)} \big)^{1 / dm^2}$.}
\item{The $s_i$ are sufficiently close to $1$ such that there exist positive numbers $r_1 > r_2 > \cdots > r_m$ (which will also be contour radii) all greater than $1$ and all less than $\max \bigcup_{i = 1}^m \big( X^{(i)} \big)^{-1}$ and $\max \bigcup_{i = 1}^m \big( Y^{(i)} \big)^{-1}$, such that $r_i s_i > r_{i + 1}$ for each integer $i \in [1, m - 1]$.}
\item{For each integer $i \in [1, m]$ and $j \in [1, d_i]$, the element $q_{i, j} \in Q$ is a complex variable with magnitude $s_i$.}
\item{The $A^{(1)}, A^{(2)}, \ldots , A^{(m - 1)}$ and $B^{(1)}, B^{(2)}, \ldots , B^{(m - 1)}$ are countably infinite sets of positive real variables whose magnitudes are all less than $s_1 r_1^{-1}$.}
\end{itemize}

\noindent Under these assumptions, we obtain the following analogue of (\hyperref[operatorfunction]{\ref*{operatorfunction}}) for $C(X; Y; Q; u)$. For the remainder of this paper, we will denote $t_{i, j}$ and $q_{i, j}$ by $t_{ij}$ and $q_{ij}$, respectively. 

\begin{prop}
\label{weightoperators} 

For each finite positive integer $u$, we have that 
\begin{flalign}
\label{differencescalar}
Z_{X, Y} C(X; Y; Q; u) &= \Bigg\langle \cdots \bigg\langle \Big\langle \Big( \displaystyle\prod_{j = 1}^{d_1} q_{1j}^n [\tilde{D}_{n; q_{1j}}^1]_{X^{(1)}} \Big) F \big(X^{(1)}; Y^{(1)}, B_{[1, u]}^{(1)} \big), F \big(B_{[1, u]}^{(1)}; A_{[1, u]}^{(1)} \big) \Big\rangle_{B_{[1, u]}^{(1)}}, \nonumber \\
& \quad \Big( \displaystyle\prod_{j = 1}^{d_2} q_{2j}^{n + u} [\tilde{D}_{n + u; q_{2j}}^1]_{ \{ X^{(2)}, A_{[1, u]}^{(1)} \} } \Big) F \big(X^{(2)}, A_{[1, u]}^{(1)}; Y^{(2)}, B_{[1, u]}^{(2)} \big) \bigg\rangle_{A_{[1, u]}^{(1)}}, 	\cdots , \nonumber \\
& \quad \Big( \displaystyle\prod_{j = 1}^{d_m} q_{mj}^{n + u} [\tilde{D}_{n + u; q_{mj}}^1]_{ \{ X^{(m)}, A_{[1, u]}^{(m - 1)} \} } \Big) F \big(X^{(m)}, A_{[1, u]}^{(m - 1)}; Y^{(m)} \big) \Bigg\rangle_{A_{[1, u]}^{(m - 1)}}. 
\end{flalign}
\end{prop}

\begin{proof}
By \hyperref[weightscalar]{Lemma \ref*{weightscalar}} and \hyperref[eigenfunction]{Proposition \ref*{eigenfunction}}, we obtain that $Z_{X, Y} C(X; Y; Q; u)$ is equal to 
\begin{flalign*}
& \displaystyle\sum_{\lambda, \mu} \Big( \displaystyle\prod_{j = 1}^{d_1} q_{1j}^n [\tilde{D}_{n; q_{1j}}^1]_{X^{(1)}} \Big) s_{\lambda^{(1)}} \big(X^{(1)} \big) \bigg( \displaystyle\prod_{i = 1}^{m - 1} \Big\langle s_{\lambda^{(i)}} \big(Y^{(i)}, B_{[1, u]}^{(i)} \big), s_{\mu^{(i)}} \big(B_{[1, u]}^{(i)} \big) \Big\rangle_{B_{[1, u]}^{(i)}} \\
& \quad \times \Big\langle \Big(\displaystyle\prod_{j = 1}^{d_{i + 1}} q_{ij}^{n + u} [\tilde{D}_{n + u; q_{ij}}^1]_{\{ X^{(i + 1)}, A_{[1, u]}^{(i)} \}} \Big) s_{\lambda^{(i+1)}} \big(X^{(i + 1)}, A_{[1, u]}^{(i)} \big), s_{\mu^{(i)}} \big(A_{[1, u]}^{(i)} \big) \Big\rangle_{A_{[1, u]}^{(i)}} \bigg) s_{\lambda^{(m)}} \big(Y^{(m)} \big) 
\end{flalign*}

\noindent where $\lambda$ is summed over $\mathbb{Y}^m$ and $\mu$ is summed over $\mathbb{Y}^{m - 1}$. Let us first sum over $\lambda^{(1)}$, while fixing $\lambda^{(2)}, \lambda^{(3)}, \ldots , \lambda^{(m)}$ and $\mu^{(1)}, \mu^{(2)} , \ldots , \mu^{(m - 1)}$. From bilinearity of the scalar product, we obtain that $Z_{X, Y} C(X; Y; Q; u)$ is equal to
\begin{flalign*}
\displaystyle\sum_{\lambda \backslash \lambda^{(1)} } & \displaystyle\sum_{\mu} \bigg\langle \displaystyle\prod_{j = 1}^{d_1} q_{1j}^n [\tilde{D}_{n; q_{1j}}^1]_{X^{(1)}} \displaystyle\sum_{\lambda^{(1)} \in \mathbb{Y}} s_{\lambda^{(1)}}  \big(X^{(1)} \big) s_{\lambda^{(1)}} \big(Y^{(1)}, B_{[1, u]}^{(1)} \big), s_{\mu^{(1)}} \big(B_{[1, u]}^{(1)} \big) \bigg\rangle_{B_{[1, u]}^{(1)}} \\
& \quad \times \bigg\langle \Big( \displaystyle\prod_{j = 1}^{d_2} q_{2j}^{n + u} [\tilde{D}_{n + u; q_{2j}}^1]_{ \{ X^{(2)}, A_{[1, u]}^{(1)} \} } \Big) s_{\lambda^{(2)}} \big(X^{(2)}, A_{[1, u]}^{(1)} \big), s_{\mu^{(1)}} \big(A_{[1, u]}^{(1)} \big) \bigg\rangle_{A_{[1, u]}^{(1)}} \\ 
& \quad \times \bigg( \displaystyle\prod_{i = 2}^{m-1} \Big\langle \Big( \displaystyle\prod_{j = 1}^{d_{i + 1}} q_{ij}^{n + u} [\tilde{D}_{n + u; q_{ij}}^1]_{ \{ X^{(i + 1)}, A_{[1, u]}^{(i)} \} } \Big) s_{\lambda^{(i+1)}} \big(X^{(i + 1)}, A_{[1, u]}^{(i)} \big), s_{\mu^{(i)}} \big(A_{[1, u]}^{(i)} \big) \Big\rangle_{A_{[1, u]}^{(i)}} \\
& \quad \times \Big\langle s_{\lambda^{(i)}} \big(Y^{(i)}, B_{[1, u]}^{(i)} \big), s_{\mu^{(i)}} \big(B_{[1, u]}^{(i)} \big) \Big\rangle_{B_{[1, u]}^{(i)}} \bigg) s_{\lambda^{(m)}} \big(Y^{(m)} \big). 
\end{flalign*}

\noindent Applying the Cauchy identity (\hyperref[sum]{\ref*{sum}}), summing over $\mu^{(1)}$ (while keeping $\lambda^{(2)}, \lambda^{(3)}, \ldots , \lambda^{(m)}$ and $\mu^{(2)}, \mu^{(3)}, \ldots , \mu^{(m - 1)}$ fixed), and using bilinearity of the scalar product yields that $Z_{X, Y} C(X; Y; Q; u)$ is equal to  
\begin{flalign*}
\displaystyle\sum_{\lambda \backslash \lambda^{(1)} } & \displaystyle\sum_{\mu \backslash \mu^{(1)}} \Bigg\langle \bigg\langle \Big( \displaystyle\prod_{j = 1}^{d_1} q_{1j}^n [\tilde{D}_{n; q_{1j}}^1]_{X^{(1)}} \Big) F \big( X^{(1)}; Y^{(1)}, B_{[1, u]}^{(1)} \big), \displaystyle\sum_{\mu^{(1)} \in \mathbb{Y}} s_{\mu^{(1)}}  \big( A_{[1, u]}^{(1)} \big) s_{\mu^{(1)}} \big(B_{[1, u]}^{(1)} \big) \bigg\rangle_{B_{[1, u]}^{(1)}}, \\
& \quad \Big( \displaystyle\prod_{j = 1}^{d_2} q_{2j}^{n + u} [\tilde{D}_{n + u; q_{2j}}^1]_{ \{ X^{(2)}, A_{[1, u]}^{(1)} \} } \Big) s_{\lambda^{(2)}} \big(X^{(2)}, A_{[1, u]}^{(1)} \big) \Bigg\rangle_{A_{[1, u]}^{(1)}} \\ 
& \quad \times \bigg( \displaystyle\prod_{i = 2}^{m-1} \Big\langle \Big( \displaystyle\prod_{j = 1}^{d_{i + 1}} q_{ij}^{n + u} [\tilde{D}_{n + u; q_{ij}}^1]_{ \{ X^{(i + 1)}, A_{[1, u]}^{(i)} \} } \Big) s_{\lambda^{(i+1)}} \big(X^{(i + 1)}, A_{[1, u]}^{(i)} \big), s_{\mu^{(i)}} \big(A_{[1, u]}^{(i)} \big) \Big\rangle_{A_{[1, u]}^{(i)}} \\
& \quad \times \Big\langle s_{\lambda^{(i)}} \big(Y^{(i)}, B_{[1, u]}^{(i)} \big), s_{\mu^{(i)}} \big(B_{[1, u]}^{(i)} \big) \Big\rangle_{B_{[1, u]}^{(i)}} \bigg) s_{\lambda^{(m)}} \big(Y^{(m)} \big). 
\end{flalign*}

\noindent Applying the Cauchy identity (\hyperref[sum]{\ref*{sum}}) again and repeating this procedure until we sum over all elements of $\lambda$ and $\mu$ yields (\hyperref[differencescalar]{\ref*{differencescalar}}). 
\end{proof}

\noindent The following lemma evaluates a nested scalar product that will appear when we apply \hyperref[operatorsmeasure]{Proposition \ref*{operatorsmeasure}} to the right side of (\hyperref[differencescalar]{\ref*{differencescalar}}.

\begin{lem}
\label{fhscalar}
Let $Z = \{ z_{ij} = z_{i, j} \}$ be a set of complex variables, where $i$ ranges from $1$ to $m$ and $j$ ranges from $1$ to $d_i$, and suppose that $|z_{ij}| = r_i$ for all $i$ and $j$. Then, 
\begin{flalign} 
\label{scalars}
\displaystyle\lim_{u \rightarrow \infty} & \Bigg\langle \cdots \bigg\langle  \Big\langle \big\langle F\big(X^{(1)}; B_{[1, u]}^{(1)}, Y^{(1)} \big) \displaystyle\prod_{k = 1}^{d_1} H_{q_{1k}} \big( B_{[1, u]}^{(1)}; \{ q_{1k}^{-1} z_{1k}^{-1} \} \big) , F\big(B_{[1, u]}^{(1)}; A_{[1, u]}^{(1)} \big) \big\rangle_{B_{[1, u]}^{(1)}}, \nonumber \\ 
& F \big( X^{(2)}, A_{[1, u]}^{(1)}; Y^{(2)}, B_{[1, u]}^{(2)} \big) \displaystyle\prod_{k = 1}^{d_2} H_{q_{2k}} \big( A_{[1, u]}^{(1)};| \{ z_{2k} \} \big) H_{q_{2k}} \big( B_{[1, u]}^{(2)}; \{ q_{2k}^{-1} z_{2k}^{-1} \} \big) \Big\rangle_{A_{[1, u]}^{(1)}}, \nonumber \\
& F \big( A_{[1, u]}^{(2)}; B_{[1, u]}^{(2)} \big) \bigg\rangle_{B_{[1, u]}^{(2)}}, \cdots , F \big( X^{(m)}, A_{[1, u]}^{(m - 1)}; Y^{(m)} \big) \displaystyle\prod_{k = 1}^{d_m} H_{q_{mk}} \big( A_{[1, u]}^{(m - 1)}; \{ z_{mk} \} \big) \Bigg\rangle_{A_{[1, u]}^{(m - 1)}}
\end{flalign}

\noindent is equal to 
\begin{flalign}
\label{scalarsproduct}
& \displaystyle\prod_{1\le h \le i \le m} F \big( X^{(h)}; Y^{(i)} \big) \displaystyle\prod_{1\le h < i \le m} \displaystyle\prod_{j = 1}^{d_h} \displaystyle\prod_{k = 1}^{d_i} \displaystyle\frac{(z_{hj} - z_{ik})(q_{hj} z_{hj} - q_{ik} z_{ik})}{(q_{hj} z_{hj} - z_{ik}) (z_{hj} - q_{ik} z_{ik})} \nonumber \\
& \quad \times \displaystyle\prod_{h = 1}^{m - 1} \displaystyle\prod_{i = h + 1}^m \displaystyle\prod_{k = 1}^{d_i} H_{q_{ik}} \big(X^{(h)}; \{ z_{ik} \} \big) \displaystyle\prod_{h = 2}^m \displaystyle\prod_{i = 1}^{h - 1} \displaystyle\prod_{j = 1}^{d_i} H_{q_{ij}} \big( Y^{(h)}; \{ q_{ij}^{-1} z_{ij}^{-1} \} \big). 
\end{flalign}
\end{lem}

\begin{proof}
Let us begin by evaluating the first scalar product appearing in the nested scalar product (\hyperref[scalars]{\ref*{scalars}}), which is 
\begin{flalign}
\label{1scalars}
\displaystyle\lim_{u \rightarrow \infty} \bigg\langle F \big( X^{(1)}; B_{[1, u]}^{(1)}, Y^{(1)} \big) \displaystyle\prod_{k = 1}^{d_1} H_{q_{1k}} \big( B_{[1, u]}^{(1)}; \{ q_{1k}^{-1} z_{1k}^{-1} \} \big) , F\big(B_{[1, u]}^{(1)}; A_{[1, v]}^{(1)} \big) \bigg\rangle_{B_{[1, u]}^{(1)}}, 
\end{flalign}

\noindent for any positive integer $v$. Using (\hyperref[equalityf]{\ref*{equalityf}}) and (\hyperref[equalityh]{\ref*{equalityh}}), we obtain that (\hyperref[1scalars]{\ref*{1scalars}}) is equal to 
\begin{flalign}
\label{1scalarexponential}
F \big(X^{(1)}; Y^{(1)} \big)\displaystyle\lim_{u \rightarrow \infty} \Bigg\langle & \exp \bigg( \displaystyle\sum_{i = 1}^\infty \displaystyle\frac{p_i \big(B_{[1, u]}^{(1)} \big)}{i} \Big( p_i \big(X^{(1)} \big) + \displaystyle\sum_{j = 1}^{d_1} \displaystyle\frac{1 - q_{1j}^i}{q_{1j}^i z_{1j}^i} \Big) \bigg), \nonumber \\ 
& \exp \bigg( \displaystyle\sum_{i = 1}^\infty \displaystyle\frac{p_i \big(A_{[1, v]}^{(1)} \big) p_i \big(B_{[1, u]}^{(1)} \big)}{i} \bigg) \Bigg\rangle_{B_{[1, u]}^{(1)}}. 
\end{flalign}

\noindent Applying \hyperref[scalarseries]{Lemma \ref*{scalarseries}}, we obtain that (\hyperref[1scalarexponential]{\ref*{1scalarexponential}}) is equal to 
\begin{flalign*}
F \big(X^{(1)}; Y^{(1)} \big) \exp \Bigg( \displaystyle\sum_{i = 1}^\infty \displaystyle\frac{p_i \big(A_{[1, v]}^{(1)} \big)}{i} \bigg( p_i \big(X^{(1)} \big) + \displaystyle\sum_{j = 1}^{d_1} \displaystyle\frac{1 - q_{1j}^i}{q_{1j}^i z_{1j}^i} \bigg) \Bigg),
\end{flalign*}

\noindent which is equal to 
\begin{flalign}
\label{fh1scalar}
F \big( X^{(1)}; Y^{(1)} \big) F \big(X^{(1)}; A_{[1, v]}^{(1)} \big) \displaystyle\prod_{j = 1}^{d_1} H_{q_{1j}} \left( A_{[1, v]}^{(1)}; \{ q_{1j}^{-1} z_{1j}^{-1} \} \right), 
\end{flalign}

\noindent due to (\hyperref[equalityf]{\ref*{equalityf}}) and (\hyperref[equalityh]{\ref*{equalityh}}). Now, let us evaluate the first two scalar products appearing in the nested scalar product (\hyperref[scalars]{\ref*{scalars}}), which is 
\begin{flalign}
\label{2scalars}
\displaystyle\lim_{u \rightarrow \infty} & \Bigg\langle \bigg\langle F\big(X^{(1)}; B_{[1, u]}^{(1)}, Y^{(1)} \big) \displaystyle\prod_{k = 1}^{d_1} H_{q_{1k}} \big( B_{[1, u]}^{(1)}; \{ q_{1k}^{-1} z_{1k}^{-1} \} \big) , F\big(B_{[1, u]}^{(1)}; A_{[1, u]}^{(1)} \big) \bigg\rangle_{B_{[1, u]}^{(1)}}, \nonumber \\
& \quad F \big( X^{(2)}, A_{[1, u]}^{(1)}; Y^{(2)}, B_{[1, v]}^{(2)} \big) \displaystyle\prod_{k = 1}^{d_2} H_{q_{2k}} \big( A_{[1, u]}^{(1)}; \{ z_{2k} \} \big) H_{q_{2k}} \big( B_{[1, v]}^{(2)}; \{ q_{2k}^{-1} z_{2k}^{-1} \} \big) \Bigg\rangle_{A_{[1, u]}^{(1)}}, 
\end{flalign}

\noindent for any positive integer $v$. Inserting (\hyperref[fh1scalar]{\ref*{fh1scalar}}) into (\hyperref[2scalars]{\ref*{2scalars}}) and applying (\hyperref[equalityf]{\ref*{equalityf}}) and (\hyperref[equalityh]{\ref*{equalityh}}), we obtain that (\hyperref[2scalars]{\ref*{2scalars}}) is equal to
\begin{flalign} 
\label{2scalarsexponential}
F \big( X^{(1)}; Y^{(1)} \big) & F \big( X^{(2)}; Y^{(2)} \big) F \big( X^{(2)}; B_{[1, v]}^{(2)} \big) \displaystyle\prod_{k = 1}^{d_2} H_{q_{2k}} \big( B_{[1, v]}^{(2)}; \{ q_{2k}^{-1} z_{2k}^{-1} \} \big) \nonumber \\
\times \displaystyle\lim_{u \rightarrow \infty} \Bigg\langle & \exp \bigg( \displaystyle\sum_{i=1}^{\infty} \displaystyle\frac{p_i \big(A_{[1, u]}^{(1)} \big)}{i} \Big( p_i \big(X^{(1)} \big) + \displaystyle\sum_{j = 1}^{d_1} \displaystyle\frac{1 - q_{1j}^i}{q_{1j}^i z_{1j}^i} \Big) \bigg), \nonumber \\
& \exp \bigg( \displaystyle\sum_{i=1}^{\infty} \displaystyle\frac{p_i \big(A_{[1, u]}^{(1)} \big)}{i} \Big( p_i \big(B_{[1, v]}^{(2)} \big) + p_i \big(Y^{(2)} \big) + \displaystyle\sum_{k = 1}^{d_2} z_{2k}^i (1 - q_{2k}^i) \Big) \bigg) \Bigg\rangle_{A_{[1, u]}^{(1)}}. 
\end{flalign}

\noindent By \hyperref[scalarseries]{Lemma \ref*{scalarseries}}, we obtain that (\hyperref[2scalarsexponential]{\ref*{2scalarsexponential}}) is equal to 
\begin{flalign*} 
& F \big( X^{(1)}; Y^{(1)} \big) F \big( X^{(2)}; Y^{(2)} \big) F \big( X^{(2)}; B_{[1, v]}^{(2)} \big) \displaystyle\prod_{k = 1}^{d_2} H_{q_{2k}} \big( B_{[1, v]}^{(2)}; \{ q_{2k}^{-1} z_{2k}^{-1} \} \big) \\
& \quad \times \exp \Bigg( \displaystyle\sum_{i = 1}^{\infty} \bigg(\displaystyle\frac{p_i \big(X^{(1)} \big) \big(p_i \big(B_{[1, v]}^{(2)} \big) + p_i \big(Y^{(2)} \big) \big)}{i} + \displaystyle\sum_{k = 1}^{d_2} \displaystyle\frac{p_i \big(X^{(1)} \big) z_{2k}^i (1 - q_{2k}^i)}{i} \\
& \quad + \displaystyle\sum_{j = 1}^{d_1} \displaystyle\frac{ \big( p_i \big(B_{[1, v]}^{(2)} \big) + p_i \big(Y^{(2)} \big) \big) (1 - q_{1j}^i)}{i q_{1j}^i z_{1j}^i} + \displaystyle\sum_{j = 1}^{d_1} \displaystyle\sum_{k = 1}^{d_2} \Big( \displaystyle\frac{z_{2k}^i q_{2k}^i}{i z_{1j}^i} - \displaystyle\frac{z_{2k}^i}{i z_{1j}^i} - \displaystyle\frac{z_{2k}^i q_{2k}^i}{i z_{1j}^i q_{1j}^i} + \displaystyle\frac{z_{2k}^i}{i q_{1j}^i z_{1j}^i} \Big) \bigg) \Bigg),
\end{flalign*}

\noindent which equals
\begin{flalign*}
& F \big( X^{(1)}; Y^{(1)} \big) F \big( X^{(1)}, Y^{(2)} \big) F \big( X^{(2)}; Y^{(2)} \big) \displaystyle\prod_{k = 1}^{d_2} H_{q_{2k}} \big(X^{(1)}; \{ z_{2k} \} \big) \displaystyle\prod_{j = 1}^{d_1} H_{q_{1j}} \big( Y^{(2)}; \{ q_{1j}^{-1} z_{1j}^{-1} \} \big) \\
& \quad \times F \big(X^{(1)}, X^{(2)}; B_{[1, v]}^{(2)} \big) \displaystyle\prod_{i = 1}^2 \displaystyle\prod_{j = 1}^{d_i} H_{q_{ij}} \big( B_{[1, v]}^{(2)}; \{ q_{ij}^{-1} z_{ij}^{-1} \} \big) \displaystyle\prod_{j = 1}^{d_1} \displaystyle\prod_{k = 1}^{d_2} \displaystyle\frac{(z_{1j} - z_{2k})(q_{1j} z_{1j} - q_{2k} z_{2k} )}{(q_{1j} z_{1j} - z_{2k}) (z_{1j} - q_{2k} z_{2k})}, 
\end{flalign*}

\noindent again due to (\hyperref[equalityf]{\ref*{equalityf}}), (\hyperref[equalityh]{\ref*{equalityh}}), and the Taylor expansion of $\log (1 - x)$; convergence of the above sums is due to the inequalities set on the $r_i$, the $s_i$, the elements of the $A^{(i)}$, and the elements of the $B^{(i)}$ (for instance, the fact that each $|z_{2k} / z_{1j} q_{1j}|$ is less than $1$ follows from the fact that $r_1 s_1 > r_2$). If $m = 2$ (in which case $B^{(2)}$ is empty), then the expression above is equal to (\hyperref[scalarsproduct]{\ref*{scalarsproduct}}), which implies the corollary. In general, we may inductively repeat the above procedure to deduce the corollary. 
\end{proof} 

\subsection{Evaluating the Scalar Product} 

\noindent Now we will apply \hyperref[operatorsmeasure]{Proposition \ref*{operatorsmeasure}} to the right side of (\hyperref[differencescalar]{\ref*{differencescalar}}) and use \hyperref[functionprocess]{Lemma \ref*{functionprocess}} and \hyperref[fhscalar]{Lemma \ref*{fhscalar}} in order to obtain a contour integral expression for $\rho_{\textbf{S}} (T)$. 

\begin{prop}
\label{contoursprocess}

We have that $\rho_{\textbf{S}} (T)$ is equal to 
\begin{flalign}
\label{correlationprocesscontour}
\displaystyle\frac{1}{(-4 \pi^2)^d} \displaystyle\oint & \cdots \displaystyle\oint \displaystyle\prod_{h = 1}^m \displaystyle\prod_{i = h}^m \displaystyle\prod_{j = 1}^{d_i} H_{q_{ij}} \big(X^{(h)}; \{ z_{ij} \} \big) \displaystyle\prod_{h = 1}^m \displaystyle\prod_{i = 1}^h \displaystyle\prod_{j = 1}^{d_i} H_{q_{ij}} \big( Y^{(h)}; \{ q_{ij}^{-1} z_{ij}^{-1} \} \big) \nonumber \\
& \quad \times \displaystyle\prod_{h = 1}^m \displaystyle\prod_{j = 1}^{d_h} \displaystyle\frac{1}{z_{hj} - q_{hj} z_{hj}} \displaystyle\prod_{1\le j < k \le d_h} \displaystyle\frac{(q_{hk} z_{hk} - q_{hj} z_{hj})(z_{hk} - z_{hj})}{(q_{hk} z_{hk} - z_{hj})(z_{hk} - q_{hj} z_{hj})} \nonumber \\
& \quad \times \displaystyle\prod_{1\le h < i \le m} \displaystyle\prod_{j = 1}^{d_h} \displaystyle\prod_{k = 1}^{d_i} \displaystyle\frac{(z_{hj} - z_{ik})(q_{hj} z_{hj} - q_{ik} z_{ik})}{(q_{hj} z_{hj} - z_{ik}) (z_{hj} - q_{ik} z_{ik})} \displaystyle\prod_{h = 1}^m \displaystyle\prod_{j = 1}^{d_h} q_{hj}^{t_{hj}} d z_{hj} d q_{hj}, 
\end{flalign}

\noindent where the contour for $z_{hj}$ is the positively oriented circle $|z_{hj}| = r_h$ and the contour for each $q_{hj}$ is the positively oriented circle $|q_{hj}| = s_h$ for each integer $h \in [1, m]$ and $j \in [1, d_h]$. 
\end{prop}

\begin{proof}
By \hyperref[functionprocess]{Lemma \ref*{functionprocess}}, $\rho_{\textbf{S}} (T)$ is the coefficient of $Q^{-T}$ of $C(X; Y; Q; \infty)$, which is equal to the limit as $u$ tends to $\infty$ of the coefficient of $Q^{-T}$ in $C(X; Y; Q; u)$, by \hyperref[limit]{Corollary \ref*{limit}}. Therefore, in order to obtain a contour integral form for $\rho_{\textbf{S}} (T)$, we will apply \hyperref[operatorsmeasure]{Proposition \ref*{operatorsmeasure}} to the right side of (\hyperref[differencescalar]{\ref*{differencescalar}}) $m$ times. The first term in the nested scalar product on the right side of (\hyperref[differencescalar]{\ref*{differencescalar}}) is 
\begin{flalign}
\label{scalar1}
\Bigg\langle \Big( \displaystyle\prod_{j = 1}^{d_1} q_{1j}^n [\tilde{D}_{n; q_{1j}}^1]_{X^{(1)}} \Big) F \big(X^{(1)}; Y^{(1)}, B_{[1, u]}^{(1)} \big), F \big(B_{[1, u]}^{(1)}; A_{[1, u]}^{(1)} \big) \Bigg\rangle_{B_{[1, u]}^{(1)}}. 
\end{flalign}

\noindent Applying \hyperref[operatorsmeasure]{Proposition \ref*{operatorsmeasure}} yields that (\hyperref[scalar1]{\ref*{scalar1}}) is equal to 
\begin{flalign}
\label{contourscalar1}
\Bigg\langle \displaystyle\frac{1}{(2\pi i)^{d_1}} \displaystyle\oint & \cdots \displaystyle\oint \displaystyle\prod_{j = 1}^{d_1} \displaystyle\frac{q_{1j}}{z_{1j} - q_{1j} z_{1j}} \displaystyle\prod_{1\le j < k\le d_1} \displaystyle\frac{(q_{1k} z_{1k}- q_{1j} z_{1j})(z_{1k} - z_{1j})}{(z_{1k} - q_{1j} z_{1j})(q_{1k} z_{1k} - z_{1j})} \nonumber \\ 
& \quad \times F\big(X^{(1)}; Y^{(1)}, B_{[1, u]}^{(1)} \big) \displaystyle\prod_{j=1}^{d_1} H_{q_{1j}} \big(X^{(1)}; \{ z_{1j} \} \big) H_{q_{1j}} \big(Y^{(1)}; \{ q_{1j}^{-1} z_{1j}^{-1} \} \big) \nonumber \\
& \quad \times \displaystyle\prod_{j = 1}^{d_1} H_{q_{1j}} \big( B_{[1, u]}^{(1)}; \{ q_{1j}^{-1} z_{1j}^{-1} \} \big) dz_{1j}, F\big(B_{[1, u]}^{(1)}; A_{[1, u]}^{(1)} \big) \Bigg\rangle_{B_{[1, u]}^{(1)}}, 
\end{flalign}

\noindent where the contour for each $z_{1j}$ is the union of the positively oriented circle $|z_{1j}| = r_1$ and the negatively oriented circle $|z_{1j}| = R$, where $R$ is any positive number greater than $\max \big( X^{(1)} \big)^{-1}$. Now, let $R$ tend to $\infty$. The expression (\hyperref[contourscalar1]{\ref*{contourscalar1}}) is the sum of $2^{d_1}$ integrals, in which each variable is integrated either along a circle of radius $r_1$ or along a circle of radius $R$. Using similar reasoning as applied in \hyperref[contourmeasure]{Proposition \ref*{contourmeasure}}, we see that any summand in which some variable $z_{1i}$ is integrated along a circle of radius $R$ contains a factor of $(z_{1i})^n$, due to the $H_{q_{1i}} \big( X^{(i)}; \{ z_{1j} \} \big)$ term. Therefore, summands of this type do not affect the coefficient of $Q^{-T}$ in $C(X; Y; Q; u)$, and we can omit them for the purpose of evaluating this coefficient. Hence, in order to obtain the coefficient of $Q^{-T}$ in $C(X; Y; Q; u)$, we may replace (\hyperref[scalar1]{\ref*{scalar1}}) in (\hyperref[differencescalar]{\ref*{differencescalar}}) with the integral 
\begin{flalign}
\label{innercontourscalar1}
\displaystyle\frac{1}{(2\pi i)^{d_1}} \displaystyle\oint & \cdots \displaystyle\oint \displaystyle\prod_{j = 1}^{d_1} \displaystyle\frac{q_{1j}}{z_{1j} - q_{1j} z_{1j}} \displaystyle\prod_{1\le j < k\le d_1} \displaystyle\frac{(q_{1k} z_{1k}- q_{1j} z_{1j})(z_{1k} - z_{1j})}{(z_{1k} - q_{1j} z_{1j})(q_{1k} z_{1k} - z_{1j})} \nonumber \\ 
& \quad \times \Bigg\langle F\big(X^{(1)}; Y^{(1)}, B_{[1, u]}^{(1)} \big) \displaystyle\prod_{j = 1}^{d_1} H_{q_{1j}} \big( B_{[1, u]}^{(1)}; \{ q_{1j}^{-1} z_{1j}^{-1} \} \big), F\big(B_{[1, u]}^{(1)}; A_{[1, u]}^{(1)} \big) \Bigg\rangle_{B_{[1, u]}^{(1)}} \nonumber \\
& \quad \times \displaystyle\prod_{j=1}^{d_1} H_{q_{1j}} \big(X^{(1)}; \{ z_{1j} \} \big) H_{q_{1j}} \big(Y^{(1)}; \{ q_{1j}^{-1} z_{1j}^{-1} \} \big) dz_{1j}, 
\end{flalign}

\noindent in which the contour for each $z_{1j}$ is the positively oriented circle $|z_{1j}| = r_1$. Here, we have used bilinearity of the scalar product to commute integration with the scalar product. 

Now let us repeat this procedure. After replacing (\hyperref[scalar1]{\ref*{scalar1}}) with (\hyperref[innercontourscalar1]{\ref*{innercontourscalar1}}) and applying \hyperref[operatorsmeasure]{Proposition \ref*{operatorsmeasure}} again, 
\begin{flalign}
\label{scalar2}
\Bigg\langle \bigg\langle \Big( \displaystyle\prod_{j = 1}^{d_1} q_{1j}^n [\tilde{D}_{n; q_{1j}}^1]_{X^{(1)}} \Big) F \big(X^{(1)}; Y^{(1)}, B_{[1, u]}^{(1)} \big), F \big(B_{[1, u]}^{(1)}; A_{[1, u]}^{(1)} \big) \bigg\rangle_{B_{[1, u]}^{(1)}}, \nonumber \\
\Big( \displaystyle\prod_{j = 1}^{d_2} q_{2j}^{n + u} [\tilde{D}_{n + u; q_{2j}}^1]_{ \{ X^{(2)}, A_{[1, u]}^{(1)} \} } \Big) F \big(X^{(2)}, A_{[1, u]}^{(1)}; Y^{(2)}, B_{[1, u]}^{(2)} \big) \Bigg\rangle_{A_{[1, u]}^{(1)}}
\end{flalign}

\noindent becomes 
\begin{flalign*}
& \Bigg\langle \displaystyle\frac{1}{(2\pi i)^{d_1}} \displaystyle\oint \cdots \displaystyle\oint \displaystyle\prod_{j = 1}^{d_1} \displaystyle\frac{q_{1j}}{z_{1j} - q_{1j} z_{1j}} \displaystyle\prod_{1\le j < k\le d_1} \displaystyle\frac{(q_{1k} z_{1k}- q_{1j} z_{1j})(z_{1k} - z_{1j})}{(z_{1k} - q_{1j} z_{1j})(q_{1k} z_{1k} - z_{1j})} \nonumber \\ 
& \qquad \qquad \times \Bigg\langle F\big(X^{(1)}; Y^{(1)}, B_{[1, u]}^{(1)} \big) \displaystyle\prod_{j = 1}^{d_1} H_{q_{1j}} \big( B_{[1, u]}^{(1)}; \{ q_{1j}^{-1} z_{1j}^{-1} \} \big), F\big(B_{[1, u]}^{(1)}; A_{[1, u]}^{(1)} \big) \Bigg\rangle_{B_{[1, u]}^{(1)}} \\ 
& \qquad \qquad \times \displaystyle\prod_{j=1}^{d_1} H_{q_{1j}} \big(X^{(1)}; \{ z_{1j} \} \big) H_{q_{1j}} \big(Y^{(1)}; \{ q_{1j}^{-1} z_{1j}^{-1} \} \big) dz_{1j}, \nonumber \\
& \displaystyle\frac{F\big(X^{(2)}, A_{[1, u]}^{(1)}; Y^{(2)}, B_{[2, u]}^{(2)} \big)}{(2 \pi i)^{d_2}} \displaystyle\oint \cdots \displaystyle\oint \displaystyle\prod_{j=1}^{d_2} H_{q_{2j}} \big(X^{(2)}, A_{[1, u]}^{(1)}; \{ z_{2j} \} \big) H_{q_{2j}} \big(Y^{(2)}, B_{[1, u]}^{(2)}; \{ q_{2j}^{-1} z_{2j}^{-1} \} \big) \nonumber \\
& \qquad \qquad \qquad \qquad \qquad \times \displaystyle\prod_{1\le j < k\le d_2} \displaystyle\frac{(q_{2k} z_{2k} - q_{2j} z_{2j})(z_{2k} - z_{2j})}{(z_{2k} - q_{2j} z_{2j})(q_{2k} z_{2k} - z_{2j})} \displaystyle\prod_{j = 1}^{d_2} \displaystyle\frac{q_{2j} dz_{2j}}{z_{2j} - q_{2j} z_{2j}} \Bigg\rangle_{A_{[1, u]}^{(1)}}, 
\end{flalign*}

\noindent where the contour for each $z_{1j}$ is the positively oriented circle $|z_{1j}| = r_1$, and the contour for each $z_{2j}$ is the union of the negatively oriented circle $|z_{2j}| = R'$ and the positively oriented circle $|z_{1j}| = r_2$ (where $R'$ is some positive number greater than $w^{-1}$, for each $w$ in $X^{(2)}$ or $A_{[1, u]}^{(2)}$). Let $R'$ tend to $\infty$ and express the above integral as the sum of $2^{d_2}$ summands in which each $z_{2j}$ is either integrated along a circle of radius $r_2$ or along a circle of radius $R'$. As previously, we can omit all summands in which some variable is integrated along the circle of radius $R'$ for the purposes of finding the coefficient of $Q^{-T}$ in $C(X; Y; Q; u)$. Therefore, we can replace (\hyperref[scalar2]{\ref*{scalar2}}) with 
\begin{flalign*}
\displaystyle\frac{1}{(2\pi i)^{d_1 + d_2}} \displaystyle\oint & \cdots \displaystyle\oint \displaystyle\prod_{h = 1}^2 \displaystyle\prod_{j = 1}^{d_h} \displaystyle\frac{q_{hj}}{z_{hj} - q_{hj} z_{hj}} \displaystyle\prod_{1\le j < k \le d_h} \displaystyle\frac{(q_{hk} z_{hk} - q_{hj} z_{hj})(z_{hk} - z_{hj})}{(q_{hk} z_{hk} - z_{hj})(z_{hk} - q_{hj} z_{hj})} \\ 
& \quad \times \Bigg\langle \bigg\langle F \big(X^{(1)}; Y^{(1)}, B_{[1, u]}^{(1)} \big) \displaystyle\prod_{j = 1}^{d_1} H_{q_{1j}} \big( B_{[1, u]}^{(1)}; \{ q_{1j}^{-1} z_{1j}^{-1} \} \big) , F \big(B_{[1, u]}^{(1)}; A_{[1, u]}^{(1)} \big) \bigg\rangle_{B_{[1, u]}^{(1)}}, \\ 
& \qquad  F \big(X^{(2)}, A_{[1, u]}^{(1)}; Y^{(2)}, B_{[1, u]}^{(2)} \big) \displaystyle\prod_{j = 1}^{d_2} H_{q_{2j}} \big(A_{[1, u]}^{(1)}; \{ z_{2j} \} \big) H_{q_{2j}} \big( B_{[1, u]}^{(2)}; \{ q_{2j}^{-1} z_{2j}^{-1} \} \big) \Bigg\rangle_{A_{[1, u]}^{(1)}} \\
& \quad \times \displaystyle\prod_{h = 1}^2 \displaystyle\prod_{j=1}^{d_h} H_{q_{hj}} \big(X^{(h)}; \{ z_{hj} \} \big) H_{q_{hj}} \big( Y^{(h)}; \{ q_{hj}^{-1} z_{hj}^{-1} \} \big) dz_{hj}, 
\end{flalign*}
\noindent where each $z_{1j}$ is integrated along the positively oriented circle $|z_{1j}| = r_1$ and each $z_{2j}$ is integrated along the positively oriented circle $z_{2j} = r_2$ (we have again commuted the scalar product with integration). 

Repeating this procedure on the other terms in the nested scalar product on the right side of (\hyperref[differencescalar]{\ref*{differencescalar}}) yields that the coefficient of $Q^{-T}$ in $Z_{X, Y} C(X; Y; Q; u)$ is equal to the coefficient of $Q^{-T}$ in 
\begin{flalign}
\label{contourprocess} 
\displaystyle\frac{1}{(2 \pi i)^d} & \displaystyle\oint \cdots \displaystyle\oint \displaystyle\prod_{h = 1}^m \displaystyle\prod_{j = 1}^{d_h} \displaystyle\frac{q_{hj}}{z_{hj} - q_{hj} z_{hj}} \displaystyle\prod_{1\le j < k \le d_h} \displaystyle\frac{(q_{hk} z_{hk} - q_{hj} z_{hj})(z_{hk} - z_{hj})}{(q_{hk} z_{hk} - z_{hj})(z_{hk} - q_{hj} z_{hj})} \nonumber \\
& \quad \times \Bigg\langle \cdots \bigg\langle \Big\langle \big\langle F\big(X^{(1)}; B_{[1, u]}^{(1)}, Y^{(1)} \big) \displaystyle\prod_{j = 1}^{d_1} H_{q_{1j}} \big( B_{[1, u]}^{(1)}; \{ q_{1j}^{-1} z_{1j}^{-1} \} \big) , F\big(B_{[1, u]}^{(1)}; A_{[1, u]}^{(1)} \big) \big\rangle_{B_{[1, u]}^{(1)}}, \nonumber \\
& \quad \quad F \big( X^{(2)}, A_{[1, u]}^{(1)}; Y^{(2)}, B_{[1, u]}^{(2)} \big) \displaystyle\prod_{j = 1}^{d_2} H_{q_{2j}} \big( A_{[1, u]}^{(1)}; \{ z_{2j} \} \big) H_{q_{2j}} \big( B_{[1, u]}^{(2)}; \{ q_{2j}^{-1} z_{2j}^{-1} \} \big) \Big\rangle_{A_{[1, u]}^{(1)}}, \nonumber \\
& \quad \quad F \big( A_{[1, u]}^{(2)}; B_{[1, u]}^{(2)} \big) \bigg\rangle_{B_{[1, u]}^{(2)}}, \cdots , F \big( X^{(m)}, A_{[1, u]}^{(m - 1)}; Y^{(m)} \big) \displaystyle\prod_{j = 1}^{d_m} H_{q_{mj}} \big( A_{[1, u]}^{(m - 1)}; \{ z_{mj} \} \big) \Bigg\rangle_{A_{[1, u]}^{(m - 1)}} \nonumber \\
& \quad \times \displaystyle\prod_{h = 1}^m \displaystyle\prod_{j = 1}^{d_h} H_{q_{hj}} \big( X^{(h)}; \{ z_{hj} \} \big) H_{q_{hj}} \big( Y^{(h)}; \{ q_{hj}^{-1} z_{hj}^{-1} \} \big) d z_{hj}, 
\end{flalign}

\noindent where the contour for $z_{ij}$ is the positively oriented circle $|z_{ij}| = r_i$ for each integer $i \in [1, m]$ and $j \in [1, d_i]$. Taking the limit at $u$ tends to $\infty$ and applying \hyperref[functionprocess]{Lemma \ref*{functionprocess}}, \hyperref[limit]{Corollary \ref*{limit}}, and \hyperref[fhscalar]{Lemma \ref*{fhscalar}} yields that $\rho_{\textbf{S}} (T)$ is equal to the coefficient of $Q^{-T}$ in 
\begin{flalign*}
\displaystyle\frac{1}{(2 \pi i)^d} \displaystyle\oint & \cdots \displaystyle\oint \displaystyle\prod_{h = 1}^m \displaystyle\prod_{i = h}^m \displaystyle\prod_{k = 1}^{d_i} H_{q_{ik}} \big(X^{(h)}; \{ z_{ik} \} \big) \displaystyle\prod_{h = 1}^m \displaystyle\prod_{i = 1}^h \displaystyle\prod_{j = 1}^{d_i} H_{q_{ij}} \big( Y^{(h)}; \{ q_{ij}^{-1} z_{ij}^{-1} \} \big) \\
& \quad \times \displaystyle\prod_{h = 1}^m \displaystyle\prod_{j = 1}^{d_h} \displaystyle\frac{q_{hj}}{z_{hj} - q_{hj} z_{hj}} \displaystyle\prod_{1\le j < k \le d_h} \displaystyle\frac{(q_{hk} z_{hk} - q_{hj} z_{hj})(z_{hk} - z_{hj})}{(q_{hk} z_{hk} - z_{hj})(z_{hk} - q_{hj} z_{hj})} \nonumber \\
& \quad \times \displaystyle\prod_{1\le h < i \le m} \displaystyle\prod_{j = 1}^{d_h} \displaystyle\prod_{k = 1}^{d_i} \displaystyle\frac{(z_{hj} - z_{ik})(q_{hj} z_{hj} - q_{ik} z_{ik})}{(q_{hj} z_{hj} - z_{ik}) (z_{hj} - q_{ik} z_{ik})} \displaystyle\prod_{h = 1}^m \displaystyle\prod_{j = 1}^{d_h} d z_{hj} 
\end{flalign*}

\noindent where the contours are as above. Then, multiplying the above expression by $Q^{T - 1}$ (where $T - 1$ consists of the elements $(i, t_{ij} - 1)$ for each integer $i \in [1, m]$ and $j \in [1, d_i]$), integrating each $q_{ij}$ along the positively oriented circle $|q_{ij}| = s_i$, and applying the residue theorem yields the proposition. 
\end{proof}

\noindent We may now establish \hyperref[process2]{Theorem \ref*{process2}}. 

\begin{proof}[Proof of Theorem 3.1.1]
Let $\textbf{M} (Q; Z)$ denote the $d \times d$ matrix, whose rows and columns are indexed by pairs of integers $(i, j)$ with $i \in [1, m]$ and $j \in [1, d_i]$ ordered lexicographically from left to right (so that row $(i, j)$ is below row $(i', j')$ if $i > i'$ or if $i = i'$ and $j > j'$, and column $(i, j)$ is to the right of column $(i', j')$ if $i > i'$ or if $i = i'$ and $j > j'$), and whose $\big( (j, k), (j', k') \big)$ entry is $1 / (z_{jk} - q_{j' k'} z_{j' k'})$ for all $j$, $j'$, $k$, and $k'$. The Cauchy determinant identity (\hyperref[determinant]{\ref*{determinant}}) implies that 
\begin{flalign}
\label{processdeterminant}
\det \textbf{M}(Q; Z) &= \displaystyle\prod_{h = 1}^m \displaystyle\prod_{j = 1}^{d_h} \displaystyle\frac{q_{hj}}{z_{hj} - q_{hj} z_{hj}} \displaystyle\prod_{1\le j < k \le d_h} \displaystyle\frac{(q_{hk} z_{hk} - q_{hj} z_{hj})(z_{hk} - z_{hj})}{(q_{hk} z_{hk} - z_{hj})(z_{hk} - q_{hj} z_{hj})} \nonumber \\
& \quad \times \displaystyle\prod_{1\le h < i \le m} \displaystyle\prod_{j = 1}^{d_h} \displaystyle\prod_{k = 1}^{d_i} \displaystyle\frac{(z_{hj} - z_{ik})(q_{hj} z_{hj} - q_{ik} z_{ik})}{(q_{hj} z_{hj} - z_{ik}) (z_{hj} - q_{ik} z_{ik})}. 
\end{flalign} 

\noindent Inserting (\hyperref[processdeterminant]{\ref*{processdeterminant}}) into \hyperref[contoursprocess]{Proposition \ref*{contoursprocess}} and using the definition (\hyperref[equalityh]{\ref*{equalityh}}) yields 
\begin{flalign*}
\rho_{\textbf{S}} (T) &= \displaystyle\frac{1}{(-4 \pi^2)^d} \displaystyle\oint \cdots \displaystyle\oint \displaystyle\prod_{i = 1}^m \displaystyle\prod_{j = 1}^{d_i} \displaystyle\prod_{1\le h \le i \le h' \le m} \displaystyle\frac{F \big( X^{(h)}; \{ z_{ij} \} \big) F \big( Y^{(h')}; \{ q_{ij}^{-1} z_{ij}^{-1} \} \big)}{F \big( Y^{(h')}; \{ z_{ij}^{-1} \} \big) F \big( X^{(h)}; \{ q_{ij} z_{ij} \} \big)} \nonumber \\ 
& \qquad \qquad \qquad \times \det \textbf{M} (Q; Z) \displaystyle\prod_{h = 1}^m \displaystyle\prod_{j = 1}^{d_h} q_{hj}^{t_{hj}} d z_{hj} d q_{hj}, 
\end{flalign*}

\noindent where the contour for $z_{hj}$ is the positively oriented circle $|z_{hj}| = r_h$ and the contour for each $q_{hj}$ is the positively oriented circle $|q_{hj}| = s_h$ for each integer $h \in [1, m]$ and $j \in [1, d_h]$. 

Setting $w_{ij} = w_{i,j} = q_{ij} z_{ij}$, it follows that 
\begin{flalign}
	\label{correlationprocessdeterminant}
	\rho_{\textbf{S}} (T) &= \displaystyle\frac{1}{(-4 \pi^2)^d} \displaystyle\oint \cdots \displaystyle\oint \displaystyle\prod_{i = 1}^m \displaystyle\prod_{j = 1}^{d_i} \displaystyle\prod_{1\le h \le i \le h' \le m} \displaystyle\frac{F \big( X^{(h)}; \{ z_{ij} \} \big) F \big( Y^{(h')}; \{ w_{ij}^{-1} \} \big)}{F \big( Y^{(h')}; \{ z_{ij}^{-1} \} \big) F \big( X^{(h)}; \{ w_{ij} \} \big)} \nonumber \\ 
	& \qquad \qquad \qquad \times \det \textbf{M}_0 (W; Z) \displaystyle\prod_{h = 1}^m \displaystyle\prod_{j = 1}^{d_h} w_{hj}^{t_{hj}} z_{hj}^{-t_{hj}-1}  d w_{hj} d z_{hj}, 
\end{flalign}

\noindent where $\textbf{M}_0 (W; Z)$ denotes the $d \times d$ matrix, whose rows and columns are indexed by pairs of integers $(i, j)$ with $i \in [1, m]$ and $j \in [1, d_i]$ (ordered lexicographically from left to right), and whose $\big( (j, k); (j', k') \big)$ entry is given by $(z_{jk} - w_{j'k'})^{-1}$ for all $j$, $j'$, $k$, and $k'$. In the above, for each integer $h \in [1, m]$ and $j \in [1, d_h]$, the contours for $z_{hj}$ and $r_{hj}$ are the positively oriented circles $|z_{hj}| = r_h$ and $|w_{hj}| = r_h s_h$, respectively. Observe for any integers $h, h' \in [1, m]$; $j \in [1, d_h]$; and $j' \in [1, d_{h'}]$ that 
\begin{flalign}
	\label{zhjwhj} 
	|z_{hj}| > |w_{h'j'}|, \qquad \text{if $h \le h'$}; \qquad \qquad |w_{h'j'}| > |z_{hj}|, \qquad \text{if $h' < h$}. 
\end{flalign}

\noindent Indeed, due to the constraints on the $(r_k, s_k)$ imposed in \hyperref[ProductW]{Section \ref*{ProductW}}, for $h \le h'$ we have $|z_{hj}| = r_h \ge r_{h'} > r_{h'} s_{h'} = |w_{h'j'}|$, and for $h' < h$ (so that $h' + 1 \le h$) we have $|w_{h'j'}| = r_{h'} s_{h'} > r_{h'+1} \ge r_h = |z_{hj}|$. 

Then expanding (\hyperref[correlationprocessdeterminant]{\ref*{correlationprocessdeterminant}}) as a signed sum, using (\hyperref[zhjwhj]{\ref*{zhjwhj}}), and deforming the contours will yield \hyperref[process2]{Theorem \ref*{process2}}; we omit this since it is similar to the proof of \hyperref[measure2]{Theorem \ref*{measure2}}. 

\end{proof}

\section{Acknowledgements}
This project was partially completed at MIT's Summer Program for Undergraduate Research (SPUR), under the funding of NSF's Research Training Grant and under the supervision of Pavel Etingof. The author heartily thanks Pavel Etingof and Guozhen Wang for their advice and conversations; Alexei Borodin for suggesting this project; Ivan Corwin for his valuable aid in revising this paper; Evgeni Dimitrov for pointing out several misprints; and the referees for their helpful suggestions.

\end{document}